\documentclass[12pt]{article}

\usepackage{amssymb,amscd,amsthm, verbatim,amsmath,color,fancyhdr, mathrsfs}
\usepackage{graphicx}
\usepackage{turnstile}
\usepackage{multicol}

\usepackage[T1]{fontenc}

\usepackage{hyperref}

\usepackage[letterpaper, left=2.5cm, right=2.5cm, top=2.5cm,
bottom=2.5cm,dvips]{geometry}

\theoremstyle{definition}
\newtheorem{thm}{Theorem}[section]
\newtheorem{prop}[thm]{Proposition}
\newtheorem{obs}[thm]{Observation}
\newtheorem{lem}[thm]{Lemma}
\newtheorem{cor}[thm]{Corollary}

\newtheorem{defi}[thm]{Definition}

\newtheorem{conj}[thm]{Conjecture}
\newtheorem{que}[thm]{Question}

\title{Constructions of graphs with any possible two-fold automorphism and automorphism groups}
\author{Bartłomiej Bychawski}
\date{}

\begin{document}

\maketitle

\begin{abstract}
    It is known that the canonical double cover of any connected nonbipartite graph have an automorphism group of the form $\text{H} \rtimes \mathbb{Z}_2$, where $\text{H}$ is the set of automorphism which preserve bipartite parts. We construct connected nonbipartite and vertex determining graphs whose canonical double covers have auromorphisms group isomorphic to any semisimple product of $\mathbb{Z}_2$ with any abstract group $\text{H}$. Later we show, that the canonical double cover of any asymmetric graph have abelian automorphisms group of odd order. The above construction provides an example of asymmetric graph for any such group. By modifying the aforementioned construction we obtain graphs which have any possible number and type of graphs with isomorphic double covers. 
\end{abstract}


\section*{History of the problem}
    In this paper we study \text{stability} of graphs. 
    We call a graph $\Gamma$ \textit{stable}, if and only if $\text{Aut}(\Gamma \times \text{K}_2) \cong \text{Aut}(\Gamma) \times \mathbb{Z}_2$, and \textit{unstable} otherwise. $K_2$ stands for a graph with two vertices and an edge between them. By $\Gamma_1 \times \Gamma_2$, where $\Gamma_1 = (V_1,E_1)$ and $\Gamma_1 = (V_2,E_2)$ we denote the graph
    $$
    \Gamma_1 \times \Gamma_2 = (V_1 \times V_2,\{((v_1,v_2),(w_1,w_2)) \mid (v_1,w_1) \in E_1 \text{ and } (v_2,w_2) \in E_2 \}).
    $$
    We will refer to $\Gamma \times \text{K}_2$ as the \textit{canonical double cover} of $\Gamma$ and we denote it by $\text{B}\Gamma$.
    \newline
    One can give simple sufficient criteria for graph $\Gamma$ to be unstable. Namely, any disconnected, bipartite or not reduced graph is known to be unstable. We call a graph \textit{reduced} if any two distinct vertices have different neighbourhoods. Reduced graphs are sometimes refered to as \textit{unworthy} or \textit{vertex determining} graphs. Unstable graph which is connected, nonbipartite and reduced is usually refered to as \textit{nontrivially unstable}. Main focus of studying stability is therefore to obtain information about nontrivially unstable graphs.
    \newline
    It is well known that for any connected, nonbipartite and reduced graph $\Gamma$ we have $\text{Aut}(\text{B}\Gamma) = \text{Aut}_0(\text{B}\Gamma) \rtimes \mathbb{Z}_2$, where $\text{Aut}_0(\text{B}\Gamma)$ is the subgroup of automorphisms preserving bipartite parts of $\text{B}\Gamma$.  It was early noticed, ( cf. \cite{zbMATH01675824}) that $\Gamma$ is stable, if and only if $\mathbb{Z}_2$ lies in the center of $\text{Aut}_0(\text{B}\Gamma) \rtimes \mathbb{Z}_2$. This points out, that it is sufficient to study group $\text{Aut}_0(B\Gamma)$ and the action of conjugacion by nontrivial element of $\mathbb{Z}_2$. This idea lead to the discovery of \textit{two-fold automorphisms}. Those can be defined for any graph $\Gamma=(V,E)$ as a set
    $$
    \{(\pi_1,\pi_2) \mid \pi_1,\pi_2 \in \text{Sym}(V) \text { and } \forall_{v,w \in V} \text{ } (v,w) \in E \Leftrightarrow (\pi_1(v),\pi_2(w)) \in E \},
    $$
    which forms a group with coordinetewise composition of functions. $\text{Sym}(V)$ stands for the set of all permutations of set $V$. It is usually referred to with symbol $\text{Aut}^{\text{TF}}(\Gamma)$. Connection of this group to the problem of stability, can be easily explained by isomorphism $\text{Aut}^{\text{TF}}(\Gamma) \cong \text{Aut}_0(\text{B}\Gamma)$.
    \newline

    Two-fold automorphisms were at first studied in \cite{zbMATH03378944} and \cite{zbMATH03353315} as autotopy of graphs in different context, and then in \cite{zbMATH05944472} and \cite{zbMATH06471221} in context of stability. In \cite{zbMATH05944472} somewhat unintuitive example of nontrivially unstable asymmetric graph was given. By asymmetric we denote graphs without nontrivial automorphisms. In \cite{zbMATH07349648} it was demonstrated that this example is in fact the smallest such graph and an infinite family of asymmetric unstable graphs was constructed. In \cite{zbMATH06719978} using ideas similar to two-fold automorphisms it was discovered that a graph $\Gamma$ can be reconstructed from a family of open vertex neighbourhoods of all vertices if and only if there are no other graphs with a canonical double cover isomorphic to $\text{B}\Gamma$.

    \section*{Content summary}
    In this paper our main subject of study are finite undirected looples graphs. Those are usually refer to as simple graphs.
    \newline
    
    In Sections \ref{SECTION: Preliminaries} and \ref{SECTION: Functions gamma and alpha of TF-Projections on reduced graphs} we show that for any reduced graph $\Gamma$, a group of permutations of $V$ which maps any neighbourhood from $\Gamma$ to some neighbourhood, is isomorphic to $\text{Aut}^{\text{TF}}(\Gamma)$ (Proposition \ref{PROP: Aut^TF jest izo z Aut^pi}). We will refer to this group as \textit{two-fold projections} and use symbol $\text{Aut}^{\pi}(\Gamma)$ to describe it. We also recall and prove some general connections between $\text{Aut}(\Gamma)$ and $\text{Aut}^{\text{TF}}(\Gamma)$. In Sections \ref{SECTION: TF-Projections as a subgroup of Automophisms of Gamma^2} and \ref{SECTION: Properties of Topological Automorphisms} we present methods for studying two-fold automorphisms, especially we introduce the idea of the topological automorphisms of a graph which is totally different from previous approaches.
    \newline
    
    In Section \ref{SECTION: Construction of (H,sigma) conected nonbipartite graphs} we present main construction of this paper, which proves that for any abstract group $\text{H} \rtimes \mathbb{Z}_2$ there exists a connected, reduced and nonbipartite graph $\Gamma$ such that $\text{Aut}(\text{B}\Gamma) \cong \text{H} \rtimes \mathbb{Z}_2$ and $\text{Aut}^{\text{TF}}(\Gamma) \cong \text{H}$ (Theorem \ref{konstrukcja (H,sigma) grafow}). Section \ref{SECTION: Clasification and construction of alpha-automorphic graphs} contains the study two-fold automorphisms of asymmetric graphs. It leads to a conclusion, that Theorem \ref{konstrukcja (H,sigma) grafow} is not true if we would ask for $\Gamma$ such that $\text{Aut}^{\pi}(\Gamma)$ acts transitively on vertices of $\Gamma$.
    \newline
    
    Sections \ref{SECTION: TF-isomorphisms of graphs via adjacency matrices} and \ref{SECTION: Notes on nonisomorphic graphs which are TF-isomorphic to a given graph} leads us to equations (Corollary \ref{COR: rownania dotyczace TF-iso grafow}) which holds for any reduced $\Gamma$
    $$
    |\text{Ant}_0(\Gamma)| = \sum_{\Gamma' \cong^{\text{TF}} \Gamma} \text{Inst}(\Gamma') \quad \text{and} \quad \frac{|\text{Ant}_0(\Gamma)|}{|\text{Aut}^{\pi}(\Gamma)|} = \sum_{\Gamma' \cong^{\text{TF}} \Gamma} \frac{1}{|\text{Aut}(\Gamma')|},
    $$
    where $\text{Ant}_0(\Gamma)$ is the set of permutations $\pi_0$ of $V$ such that for any $v,w \in V$ $(v,w) \in E \Leftrightarrow (\pi_0(v),{\pi_0}^{-1}(w)) \in E$ and $(v,\pi_0(v)) \notin E$ holds. Moreover, relation $\cong^{\text{TF}}$ which appears in above sums is equivalent to having isomorphic canonical double covers and sums are taken over all unlabeled graphs with that property.
    \newline
    
    In Section \ref{SECTION: Construction of (H,sigma)-graphs with any number and type of TF-isomorphic graphs} we modify construction of $\Gamma$ from Theorem \ref{konstrukcja (H,sigma) grafow} in such a way that we can choose which elements from $\text{Ant}(\Gamma)$ produce loopless graphs with canonical double covers isomorphic to $\text{B}\Gamma$. In Section \ref{SECTION: Further notes on nonisomorphic graphs which are TF-isomorphic to a given graph} we give some additional remarks about number and type of graphs with isomorphic canonical double covers. We then construct an infinite family of stable GRRs (Cayley graphs with $\text{Aut}(\text{Cay}(G,S)) \cong G$) which have almost the same number of such graphs as vertices. Moreover, all those graphs are stable, which is a rare case.

\section{Preliminaries} \label{SECTION: Preliminaries}


Before we will present main results of this paper we have to establish some basic definitions and conventions.

\begin{defi}
    We call an ordered pair $\Gamma = (V,E)$ a (simple and finite) graph if and only if $V$ is a finite set and $E$ is is a set of unordered pairs of different elements form $V$.
\end{defi}

\begin{defi}
    Let $\Gamma$ be a graph, and $v \in V$ be a vertex. Then we define neighbourhood of $v$ as follows
    $$
    N(v) = \{ w \mid (v,w) \in E \}.
    $$
    In situations, where one element can be a vertex of more than one graph, we write $N_{\Gamma}(v)$ for clarity. We denote the set of all neighbourhoods of a given graph $\Gamma$ by $N(\Gamma)$.
    $$
    N(\Gamma) = \{ N(v) \mid v \in V \}.
    $$
\end{defi}

Those definitions will be useful, as we will describe groups of two-fold automorphisms in terms of neighbourhoods. Now, just for completeness of this section we will restate definition of a reduced graph.

\begin{defi}
    Let $\Gamma$ be a graph. We call $\Gamma$ reduced (also known in literature as "twin free" or "worthy") if the following condition holds
    $$
    N(v) = N(w) \quad \Longrightarrow \quad v = w.
    $$
\end{defi}

Let us now define two groups related to two-fold automorphisms, namely two-fold projections and topological automorphisms, along with proving some basic results about them.

\begin{defi}
    Let $\Gamma$ be reduced graph. We call a bijection $\pi_0$ from $V$ onto itself a \textit{two-fold projection}, if and only if the following holds
    $$
    \forall_{v \in V} \quad \exists_{w \in V} \quad \pi_{0}(N(v)) = N(w).
    $$
    Therefore it is easy to see that the set of all two-fold projections of $\Gamma$ form a group with composition of functions. We denote this group by $\text{Aut}^{\pi}(\Gamma)$.
\end{defi}

\begin{obs}
    For every reduced graph $\Gamma$ we have $\text{Aut}(\Gamma) \leq \text{Aut}^{\pi}(\Gamma)$.
\end{obs}
\begin{proof}
    It is enough to notice that if $\pi_0 \in \text{Aut}(\Gamma)$, then $\pi_{0}(N(v)) = N(\pi_0(v))$.
\end{proof}


\begin{prop} \label{PROP: Aut^TF jest izo z Aut^pi}
    Let $\Gamma = (V,E)$ be a reduced graph. Then $\text{Aut}^{\text{TF}}(\Gamma) \cong \text{Aut}^{\pi}(\Gamma)$ and 
    $$
    \text{Aut}^{\pi}(\Gamma) = \{ \pi_1 \in \text{Sym}(V) \mid (\pi_1,\pi_2) \in \text{Aut}^{\text{TF}}(\Gamma)\}. 
    $$
\end{prop}
\begin{proof}
    Let us consider a graph $\text{B}\Gamma$, where vertices $V \times \{ 0 \}$ correspond to vertices of $\Gamma$ and $V \times \{ 1\}$ correspond to neighbourhoods. Clearly $((u,0),(v,1)) \in E$ if and only if $u \in N(v)$. Therefore every automorphism of $\text{B}\Gamma$ which fixes vertex set $V \times \{ 0 \}$ corresponds to exactly one element of $\text{Aut}^{\pi}(\Gamma)$ just by correspondence $(v,0) \mapsto v$. This is a bijection, because the kernel of this homomorphism is a set of all automorphisms of $\text{B}\Gamma$ which fix $V \times \{ 0 \}$ pointwise, and it consists only of the identity, as $\Gamma$ is reduced.
    \newline

    On the other hand $(\pi_1,\pi_2) \in \text{Aut}^{\text{TF}}$ if and only if a fuction given by a formula
    $$
    (u,0) \mapsto (\pi_1(u),0) \quad (v,1) \mapsto (\pi_2(v),1), 
    $$
    is an automorphism of $\text{B}\Gamma$. Now the proof is complete.
\end{proof}
Theorem above shows, that the difference between $\text{Aut}^{\text{TF}}(\Gamma)$ and $\text{Aut}^{\pi}(\Gamma)$ is strictly formal, however it will be easier for us to formulate certain results. For example it is easier to talk about orbits of $\text{Aut}^{\pi}(\Gamma)$, when we will mean the subset of vertices of $\Gamma$, rather than subset of $V \times V$.
\newline
    We will now present some basic definitions and one observation which will connect ideas of topological automorphisms and two-fold projections. This link will be further explored in Section \ref{SECTION: Properties of Topological Automorphisms}.
\begin{defi}
    Let $\Gamma$ be a graph, and $v \in V$ be a vertex. We define a ball of radius 1 with center $v$ as follows
    $$
    B(v,1) = \{ w \mid (v,w) \in E \} \sqcup \{ v \}.
    $$
     In situations, where one element can be a vertex of more than one graph we write $B_{\Gamma}(v,1)$ for clarity.
\end{defi}

\begin{defi}
    Let $\Gamma$ be graph. We call a bijection $\pi$ from $V$ onto itself a topological automorphism, if and only if the following holds
    $$
    \forall_{v \in V} \quad \exists_{w \in V} \quad \pi(B(v,1)) = B(w,1).
    $$
    Similarly as before the set of all topological automorphism of $\Gamma$ form a group with composition of functions. We denote it by $\text{Aut}^{\tau}(\Gamma)$.
\end{defi}

\begin{defi}
    Let $\Gamma = (V,E)$ be a graph. By $\Gamma^{C}$ we denote the \textit{complement of a graph} $\Gamma$, that is a graph with set of vertices equal to $V$ and edges between $v$ and $w$ if and only if $(v,w) \notin E$.
\end{defi}

\begin{obs}
    For a reduced graph $\Gamma$ we have $\text{Aut}^{\tau}({\Gamma}^C) = \text{Aut}^{\pi}(\Gamma)$.
\end{obs}
\begin{proof}
    Firstly note that
    $
    V \text{ } \backslash \text{ } N_{\Gamma}(v) = B_{{\Gamma}^C}(v,1)
    $.
    Secondly, any permutation $\pi_0: V \longrightarrow V$ gives a permutation of a family $N(\Gamma) = \{ N(v) \mid v \in V \}$ exactly when it permutes family constructed from complements of all sets from $N(\Gamma)$.
\end{proof}

\section{Functions $\gamma$ and $\alpha$ of TF-Projections on reduced graphs} \label{SECTION: Functions gamma and alpha of TF-Projections on reduced graphs}

As previously mentioned, for any connected reduced and nonbipartite graph $\Gamma$ it is true that $\text{Aut}(\text{B}\Gamma) \cong \text{Aut}^{\pi}(\Gamma) \rtimes \mathbb{Z}_2$. We start this Section by defining the action of conjugation by nontrivial element of $\mathbb{Z}_2$ on $\text{Aut}^{\pi}(\Gamma)$ in the language of neighbourhoods.

\begin{defi} \label{definicja funkcji gamma}
    Let $\Gamma$ be a reduced graph. Lets define $s$ to be the bijection between the set $V$ and the set $N(\Gamma)$ of neigbourhoods of $\Gamma$ given by $V \ni v \mapsto N(v) \in N(\Gamma)$.
\newline
    Define $\psi$ as a function from $\text{Aut}^{\pi}(\Gamma)$ to $\text{Sym}(N(\Gamma))$ given by the formula 
    $$
    \forall_{\pi \in \text{Aut}^{\pi}(\Gamma)} \quad \forall_{v \in V} \quad \psi( \pi ) (N(v)) = \pi(N(v)). 
    $$
    Let $\gamma$ denote the function from $\text{Aut}^{\pi}(\Gamma)$ to $\text{Sym}(V)$ given by the formula
    $$
     \forall_{\pi \in \text{Aut}^{\pi}(\Gamma)} \quad \gamma (\pi) = s^{-1} \circ \psi(\pi) \circ s.
    $$
    In situations, where it is not clear, which graph are we referring too we denote this function by $\gamma_{\Gamma}$.
\end{defi}

Function $\gamma$ is in fact a conjugation by a nontrivial element of $\mathbb{Z}_2$, which we will now prove. For any function $f: X \rightarrow X$ we will denote $\text{Fix}(f) = \{ x\in X \mid f(x) = x \}$. 
    
\begin{thm} \label{THM: kluczowe własności gammy}
 Let $\Gamma$ be a reduced graph. Then $\gamma(\text{Aut}^{\pi}(\Gamma)) = \text{Aut}^{\pi}(\Gamma)$ and $\gamma \in \text{Aut}(\text{Aut}^{\pi}(\Gamma))$. Furthermore $\gamma^2 \equiv Id$ and $\text{Fix}(\gamma) = \text{Aut}(\Gamma)$, and the following equivalence holds
 \begin{align} \label{ROWNANIE: gamma a krawedzie}
     \forall_{\pi \in \text{Aut}^{\pi}(\Gamma)} \quad \forall_{v,w \in V} \quad (v,w) \in E \quad \Longleftrightarrow \quad (\pi(v),\gamma(\pi)(w)) \in E.
 \end{align}
\end{thm}
\begin{proof}
    If $\pi_1 \in \text{Aut}^{\pi}(\Gamma)$, then there exists exactly one $\pi_2$ such that $(\pi_1,\pi_2) \in \text{Aut}^{\text{TF}}(\Gamma)$. Uniqueness of $\pi_2$ follows from the fact that $\Gamma$ is reduced. Let consider a graph $\text{B}\Gamma$. Since for every $v \in V$ vertex $(v,1)$ in $\text{B}\Gamma$ is connected to $N_{\Gamma}(v) \times \{ 0\}$, permutation $\pi_2$ is basically the permutation of neighbourhoods given by $\pi_1$. This is because function $s$ corresponds to a map $(v,1) \mapsto N_{\Gamma}(v) \times \{ 0\}$, and then following the Definition \ref{definicja funkcji gamma} with the fact that $\Gamma$ is reduced finally gives us $\pi_2 = \gamma(\pi_1)$.
    \newline
    $(\pi_1,\pi_2) \in \text{Aut}^{\text{TF}}(\Gamma)$ implies  $(\pi_2,\pi_1) \in \text{Aut}^{\text{TF}}(\Gamma)$ by symmetry in definition of $\text{TF}(\Gamma)$. We therefore conclude that $\gamma(\text{Aut}^{\pi}(\Gamma)) = \text{Aut}^{\pi}(\Gamma)$, so $\gamma$ is a bijection. It is a homomorphism because if $(\pi_1,\pi_2),(\tau_1,\tau_2) \in \text{Aut}^{\text{TF}}(\Gamma)$, then also $(\pi_1 \circ \tau_1,\pi_2 \circ \tau_2) \in \text{Aut}^{\text{TF}}(\Gamma)$, which gives us $\gamma(\pi_1) \circ \gamma(\tau_1) = \pi_2 \circ \tau_2 = \gamma(\pi_1 \circ \tau_1)$, and therefore $\gamma \in \text{Aut}(\text{Aut}^{\pi}(\Gamma))$. The fact that both $(\pi_1,\pi_2)$ and $(\pi_2,\pi_1)$ are elements of $\text{Aut}^{\text{TF}}(\Gamma)$ implies
    $\gamma(\gamma(\pi_1)) = \gamma(\pi_2) = \pi_1$, which proves $\gamma^2 \equiv Id$.
    \newline
    To prove the next part of this theorem, let us take one more look at $\text{B}\Gamma$. If $\pi_1 \in \text{Aut}(\Gamma)$, then $(\pi_1,\pi_1) \in \text{Aut}^{\text{TF}}(\Gamma)$, so $\gamma(\pi_1) = \pi_1$. If on the other hand $\gamma(\pi_1) = \pi_1$, then let $(u,v) \in E_{\Gamma}$. It means that $((u,0),(v,1)) \in E_{\text{B}\Gamma}$. Since $(\pi_1,\pi_1) \in \text{Aut}^{\text{TF}}(\Gamma)$, it gives us $((\pi_1(u),0),(\pi_1(v),1)) \in E_{\text{B}\Gamma}$, and from this we conclude that $(\pi_1(u),\pi_1(v)) \in E_{\Gamma}$. Therefore $\pi_1 \in \text{Aut}(\Gamma)$, and finally $\text{Fix}(\gamma) = \text{Aut}(\Gamma)$.
    \newline
    The equivalence (\ref{ROWNANIE: gamma a krawedzie}) follows from the fact that if $\pi \in \text{Aut}^{\pi}(\Gamma)$, then $(\pi,\gamma(\pi)) \in \text{Aut}^{\text{TF}}(\Gamma)$ as discussed above.   
\end{proof}

\begin{defi} \label{definicja funkcji alpha}
    Let $\Gamma$ be a reduced graph. Let $\alpha$ denote the function from $\text{Aut}^{\pi}(\Gamma)$ to $\text{Aut}^{\pi}(\Gamma)$ given by the formula
    $$
    \alpha(\pi) = \pi^{-1} \circ \gamma(\pi).
    $$
    In situations, where it is not clear, which graph are we referring too we denote this function by $\alpha_{\Gamma}$.
\end{defi}

\begin{defi}
     Let $\Gamma$ be any graph. By \textit{instability index} of $\Gamma$ we understand the number 
     $$
     \text{Inst}(\Gamma) = \frac{|\text{Aut}(B\Gamma)|}{2\cdot |\text{Aut}(\Gamma)|}.
     $$
\end{defi}

Theorem bellow establishes a link between $\alpha$ and instability index which will become important in Section \ref{SECTION: Notes on nonisomorphic graphs which are TF-isomorphic to a given graph}.

\begin{thm} \label{THM: alpha jest stala na prawostronnych warstwach Aut}
    Let $\Gamma$ be a reduced graph and $\pi_1,\pi_2 \in \text{Aut}^{\pi}(\Gamma)$ be any TF-projections. Then 
    $$
    \alpha(\pi_1) = \alpha(\pi_2) \quad \Longleftrightarrow \quad \text{Aut}(\Gamma) \pi_1 = \text{Aut}(\Gamma) \pi_2,
    $$
    and consequently $|\text{Im}(\alpha)| = \frac{|\text{Aut}^{\pi}(\Gamma)|}{|\text{Aut}(\Gamma)|}$. Moreover if $\Gamma$ is connected and nonbipartite, $|\text{Im}(\alpha)|$ is equal to instability index if $\Gamma$.
\end{thm}
\begin{proof}
    Let us start by taking $\tau \in \text{Aut}^{\pi}$ such that $\pi_2 = \tau \circ \pi_1$. Then
    $$
    \alpha(\pi_1) = \alpha(\pi_2) \quad \Longleftrightarrow \quad {\pi_1}^{-1} \gamma(\pi_1) = {\pi_1}^{-1} \tau^{-1} \gamma(\tau) \gamma(\pi_1) \quad  \Longleftrightarrow \quad e = \tau^{-1} \gamma(\tau) \quad \Longleftrightarrow 
    $$
    $$
    \Longleftrightarrow \quad \tau = \gamma(\tau) \quad \Longleftrightarrow \quad \tau \in \text{Aut}(\Gamma) \quad \Longleftrightarrow \quad \text{Aut}(\Gamma) \pi_1 = \text{Aut}(\Gamma) \pi_2.
    $$
    This implies, that every right coset of $\text{Aut}(\Gamma)$ gives a distinct value under a function $\alpha$, so in fact $|\text{Im}(\alpha)| 
 = \frac{|\text{Aut}^{\pi}(\Gamma)|}{|\text{Aut}(\Gamma)|}$.
    If $\Gamma$ is conected and nonbipartite, then either $V \times \{ 0 \}$ is fixed setwise, or its image is $V \times \{ 1 \}$, and there exist an automorphism of $\text{B}\Gamma$ of the form $(v,0) \longleftrightarrow (v,1)$, hence the last part.
\end{proof}

If a graph $\Gamma$ is nontrivially unstable, then there exists a nontrivial element in $\text{Im}(\alpha)$. For that reason it is good to know properties of elements from this set.

\begin{obs} \label{OBS wlasnosci alphy i elementow odwracanych przez gamme}
     Let $\Gamma$ be a reduced graph. Let $\tau_0 \in \text{Aut}^{\pi}(\Gamma)$, $\pi_1 \in \text{Im}(\alpha)$ and $\pi_2,\pi_3 \in \text{Aut}^{\pi}(\Gamma)$ be such that $\gamma(\pi_2) = \pi_2^{-1}$ and $\gamma(\pi_3) = \pi_3^{-1}$. Moreover, let $n \in \mathbb{Z}$ be any integer. Then: \newline
     (a) $\gamma(\pi_1) = \pi_1^{-1}$, \newline
     (b) $\gamma(\alpha(\tau_0)) = \alpha(\gamma(\tau_0))$, \newline
     (c) $\gamma(\pi_2 \pi_3 \pi_2) = (\pi_2 \pi_3 \pi_2)^{-1}$, \newline
     (d) $\pi_2 \pi_1 \pi_2 \in \text{Im}(\alpha)$, \newline
     (e) $\pi_1^n \in \text{Im}(\alpha)$.
\end{obs}
\begin{proof}
    Notice that there exists $\tau \in \text{Aut}^{\pi}(\Gamma)$ such that $\pi_1 ={\tau}^{-1} \gamma(\tau)$. (a) follow from
    $$
    \gamma(\pi_1) = {\gamma(\tau)}^{-1} \gamma(\gamma(\tau)) = {\gamma(\tau)}^{-1} {(\tau^{-1})}^{-1} = {(\tau^{-1} \gamma(\tau))}^{-1} = {\pi_1}^{-1}.
    $$
    (b) follow from calculation 
    $$
    \gamma(\alpha(\tau_0)) = \gamma({\tau_0}^{-1} \gamma(\tau_0)) = {\gamma(\tau_0)}^{-1} \gamma(\gamma(\tau_0)) = \alpha(\gamma(\tau_0)).
    $$
    (c) follows from the fact that $\gamma \in \text{Aut}(\text{Aut}^{\pi}(\Gamma))$. (d) follows from calculation
    $$
    \alpha(\tau {\pi_2}^{-1}) = {({\pi_2}^{-1})}^{-1} \tau^{-1} \gamma(\tau) \gamma({\pi_2}^{-1}) = \pi_2 \pi_1 {\gamma(\pi_2)}^{-1} = \pi_2 \pi_1 \pi_2.
    $$
    The last one is a little bit more complicated. Firstly, let's observe $\alpha(\pi_1) = {\pi_1}^{-1} \gamma(\pi_1) = {\pi_1}^{-2}$. Applying (b) with $\tau_0 = \pi_1$ gives us 
    $$
    {\pi_1}^{2} = {({\pi_1}^{-2})}^{-1} =  {\alpha(\pi_1)}^{-1} = \gamma(\alpha(\pi_1)) = \alpha(\gamma(\pi_1)) \in \text{Im}(\alpha).
    $$
    We proved the fifth claim for $n = 1,2$. Now using (a) and (d) we can proceed inductively on $n$ since
    $$
    {\pi_1}^n = \pi_1 {\pi_1}^{n-2} \pi_1 \in \text{Im}(\alpha),
    $$
    because $\gamma(\pi_1) = {\pi_1}^{-1}$ and by induction we know that ${\pi_1}^{n-2} \in \text{Im}(\alpha)$.
\end{proof}

Next corollary helps with understanding the shape of a graph which admits a nontrivial permutation $\pi_0$ inside $\text{Im}(\alpha)$. A coclique is an empty graph, which is a graph which does not contain any edges. 

\begin{cor} \label{COR: brak krawedzi w orbitach alphy}
    Let $\Gamma = (V,E)$ be a reduced graph, and let $\pi_1 \in \text{Im}(\alpha)$. Then the graph induced by any orbit of $\pi_1$ on $V$ is a coclique.
\end{cor}
\begin{proof}
    Let us prove that if $\pi_0 \in \text{Im}(\alpha)$, then $(v,\pi_0(v)) \notin E$. If we put $\pi_0 = \tau^{-1} \gamma(\tau)$, $w = \gamma(\tau)(v)$ and now notice that putting $\pi = \tau^{-1}$ in the last thesis of Theorem \ref{THM: kluczowe własności gammy} gives us
    $$
    (w,w) \in E \quad \Longleftrightarrow \quad (\tau^{-1}(w),\gamma(\tau)(w)) = (\pi_0(v),v) \in E.
    $$
    However, since edges have two different ends, this proves that $(\pi_0(v),v) \notin E$.
    \newline
    Two vertices $v$ and $w$ are in the same orbit of $\pi_0$ is the same as saying that there exists $k \in \mathbb{Z}$ such that ${\pi_0}^k(v) = w$, however last point of Observation \ref{OBS wlasnosci alphy i elementow odwracanych przez gamme} assures us, that ${\pi_0}^k \in \text{Im}(\alpha)$ for any $k$, therefore proving our thesis.
 
\end{proof}

We may now present general result about relation between cardinalities of groups $\text{Aut}(\Gamma)$ and $\text{Aut}^{\pi}(\Gamma)$.

\begin{thm} \label{THM: parzystosc H a parzystosc Fix(sigma)}
    Let $\text{H}$ be a finite group, and $\sigma \in \text{Aut}(\text{H})$ be such that $\sigma^2 \equiv Id$. Then 
    $$
    2 \mid [\text{H}:\text{Fix}(\sigma)]\quad \Longrightarrow \quad 2 \mid |\text{Fix}(\sigma)|.
    $$
\end{thm}
\begin{proof}
    Let $\mathcal{A}$ be the set of right cosets of $\text{Fix}(\sigma)$ in $\text{H}$. Observe, that $\sigma$ induces a permutation of $\mathcal{A}$. Since $\mathcal{A}$ have an even number of elements, $\text{Fix}(\sigma)$ is mapped on itself and $\sigma^2 \equiv Id$, there have to exist $h \in \text{H} \text{ } \backslash \text{ } \text{Fix}(\sigma)$ such that $\sigma(\text{Fix}(\sigma) h) = \text{Fix}(\sigma) h$. This gives us $f \in \text{Fix}(\sigma)$ such that
    $$
    f h = \sigma(h) \quad \Longleftrightarrow \quad f = h^{-1} \sigma(h),
    $$
    therefore, similarly as in Observation \ref{OBS wlasnosci alphy i elementow odwracanych przez gamme} we have $f = \sigma(f) = f^{-1}$, so $f^2 = e$. However $f \neq e$, because then $h = \sigma(h)$, which gives a contradiction. To finish the proof, recal that the order of an element divides order of a group, hence $2 \mid |\text{Fix}(\sigma)|$.
\end{proof}

\begin{cor} \label{COR: parzystosc Aut taka sama jak parzystosc Aut^pi}
    Let $\Gamma$ be a reduced graph. Then $ 2 \mid |\text{Aut}(\Gamma)| \Longleftrightarrow 2 \mid |\text{Aut}^{\pi}(\Gamma)|$.
\end{cor}
\begin{proof}
    Since $\text{Aut}(\Gamma) \leq \text{Aut}^{\pi}(\Gamma)$, then it is obvious that $ 2 \mid |\text{Aut}(\Gamma)| \Longrightarrow 2 \mid |\text{Aut}^{\pi}(\Gamma)|$. \\
    We will prove the converse by contrary. Suppose that $2 \mid |\text{Aut}^{\pi}(\Gamma)|$ and $2 \nmid |\text{Aut}(\Gamma)|$. Since $|\text{Aut}^{\pi}(\Gamma)| = |\text{Aut}(\Gamma)| \cdot [\text{Aut}^{\pi}(\Gamma):\text{Aut}(\Gamma)]$ we conclude that $2 \mid [\text{Aut}^{\pi}(\Gamma):\text{Aut}(\Gamma)]$.
    Now letting $\text{H} = \text{Aut}^{\pi}(\Gamma)$ and $\sigma = \gamma$ in Theorem \ref{THM: parzystosc H a parzystosc Fix(sigma)} and using the fact that $\text{Fix}(\gamma) = \text{Aut}(\Gamma)$ from Theorem \ref{THM: kluczowe własności gammy} we get that $2 \mid |\text{Aut}(\Gamma)|$, which gives us a contradiction.
\end{proof}

The above result clearly demonstrate, that the key features that can be used to produce algebraic results about relation between $\text{Aut}(\Gamma)$ and $\text{Aut}^{\pi}(\Gamma)$ for arbitrarily chosen connected reduced and nonbipartite graph $\Gamma$ are based on existence of $\gamma$ and Theorem \ref{THM: kluczowe własności gammy}. In Section \ref{SECTION: Construction of (H,sigma) conected nonbipartite graphs} we will show that indeed, those are the only conditions on pair of groups $G_1 \leq G_2$ necessary for existence of a reduced graph $\Gamma$ such that $G_1 \cong \text{Aut}(\Gamma)$ and $G_2 \cong \text{Aut}^{\pi}(\Gamma)$.

\section{TF-Projections as a subgroup of Automophisms of $\Gamma^2$} \label{SECTION: TF-Projections as a subgroup of Automophisms of Gamma^2}

This section is dedicated to a group which contains $\text{Aut}^{\pi}(\Gamma)$ as a subgroup and can therefore help with determining wherever some graphs are stable. 

\begin{defi}
    Let $\Gamma = (V,E)$ be a graph. Then by $\Gamma^2$ we denote the graph with vertices $V$ and edges between $v_1$ and $v_2$ if and only if there exists $u \in V$ such that $(v_1,u),(u,v_2) \in E$.
\end{defi}

We may now state the first result of this Section.

\begin{thm} \label{THM: Aut^pi(Gamma) < Aut(Gamma^2)}
    Let $\Gamma = (V,E)$ be a reduced graph. Then $\text{Aut}^{\pi}(\Gamma) \leq \text{Aut}(\Gamma^2)$.
\end{thm}
\begin{proof}
    Let $\pi_0 \in \text{Aut}^{\pi}(\Gamma)$, and $(v,w) \in E_{\Gamma^2}$. From the definition of $\Gamma^2$ there exist $u \in V$ such that $v,w \in N_{\Gamma}(u)$. Therefore, from the definition of $\text{Aut}^{\pi}(\Gamma)$ there exists $u' \in V$ such that $\pi_0(v),\pi_0(w) \in N_{\Gamma}(u')$. This shows that $(\pi_0(v),\pi_0(w)) \in E_{\Gamma^2}$. This ends the proof, as $\Gamma$ is finite, and it has finitely many edges. This ensures, that $\pi_0 \in \text{Aut}(\Gamma^2)$. 
\end{proof}

\begin{defi}
    Let $\Gamma = (V,E)$ be a finite simple graph. We define $d: V \times V \longrightarrow \mathbb{Z}_{\geq 0} \sqcup \{ \infty \}$ to be a function that outputs the length of the shortest path between two given vertices. We define this function to be $0$ if both vertices are the same, and $\infty$ if there \text{does not exist} a path between these vertices. 
\end{defi}

It is well known, that if $\Gamma$ is a connected graph, then $d(\cdot,\cdot)$ is a metric on $\Gamma$. By diameter of a connected graph $\Gamma$ we understand the maximum of the function $d(v,w)$ over all $v,w \in V$. One can easily check, that most graphs have diameter equal to $2$.

\begin{thm}
    Let $\Gamma = (V,E)$ be a reduced nonbipartite graph of diameter $2$ which does not contain any triangles. Then $\Gamma$ is stable.
\end{thm}
\begin{proof}
    We will prove that $\Gamma^2 = \Gamma^{C}$. This is enough, because stability of $\Gamma$ is equivalent to equality $\text{Aut}(\Gamma) = \text{Aut}^{\pi}(\Gamma)$, and Theorem \ref{THM: Aut^pi(Gamma) < Aut(Gamma^2)} implies $\text{Aut}(\Gamma) \leq \text{Aut}^{\pi}(\Gamma) \leq \text{Aut}(\Gamma^2) = \text{Aut}(\Gamma^{C}) = \text{Aut}(\Gamma)$.
    \newline
    Let $(v,w) \in E$. If there is a vertex $u \in V$ such that $(v,u),(w,u) \in E$, then $\Gamma$ would contain a triangle. Therefore, there is no such vertex and $(v,w) \notin E_{\Gamma^2}$. Now, let $(v,w) \notin E$. Since $\text{diam} (\Gamma) = 2$, there must exist such $u \in V$ that $(v,u),(w,u) \in E$, and therefore $(v,w) \in E_{\Gamma^2}$.
\end{proof}

\begin{cor}
    Petersen Graph is stable.
\end{cor}

Next theorem does not solve problem of stability directly, but it reduces it to the study of the group $\text{Aut}(\Gamma^2)$ for certain family of reduced graphs $\Gamma$.

\begin{thm} \label{THM: brak cykli 3,6 i podwierzcholkow implikuje ze Aut^pi(Gamma) = Aut(Gamma^2)}
    Let $\Gamma = (V,E)$ be a reduced graph which does not contain any triangles nor cycles of length six and assume additionally there are no vertices $u,v \in V$ such that $N(u) \subsetneq N(v)$. Then $\text{Aut}^{\pi}(\Gamma) = \text{Aut}(\Gamma^2)$.
\end{thm}
\begin{proof}
    Firstly we show, that there exist a bijection between maximal cliques in $\Gamma^2$ and vertices of $\Gamma$. Let us start by labeling every edge in $\Gamma^2$. One edge can have multiple labels, which are given according to the rule that edge $(v,w) \in E_{\Gamma^2}$ is labeled by $u \in V$ if and only if $v,w \in N(u)$.  We will call clique $K$ monochromatic, if there exists some vertex $u$, such that every edge inside $K$ is labeled by $u$. We will start by showing that every clique in $\Gamma^2$ is monochromatic.
    \newline
    Let $K_0$ be a clique in $\Gamma^2$ which contains at least two vertices. Let $\mathcal{K}$ be a family of monochromatic cliques contained in $K_0$, and $K_1$ be a maximal element of $\mathcal{K}$ with respect to the inclusion relation. Define $k_1$ be a vertex which labels every edge inside $K_1$ and take $v_1$ to be any element form $K_1$. If $K_1 \neq K_0$, then there exist $v_2 \in K_0 \backslash K_1$. Let $K_2$ be a maximal element of $\mathcal{K}$ which contains both $v_1$ and $v_2$. Let $k_2$ be a vertex which labels every edge inside $K_2$. Since $K_1$ was maximal, we get $k_1 \neq k_2$. Maximality of $K_1$ also means, that there exist a vertex $v_3 \in K_1 \backslash K_2$.
    Let $k_3$ be a vertex which labels edge $(v_2,v_3) \in E_{\Gamma^2}$. If $k_3 = k_1$, then $v_2 \in N_{\Gamma}(k_1)$, hence $K_1 \sqcup \{v_2\} \in N_{\Gamma}(k_1)$, and therefore $K_1 \sqcup \{v_2\}$ is a monochromatic clique inside $K$, which contradicts maximality od $K_1$. Similarly, if $k_3 = k_2$, then $v_3 \in N_{\Gamma}(k_2)$ and this time we get a contradiction with maximality of $K_2$. Therefore $k_3 \neq k_1$ and $k_3 \neq k_2$.
    \newline
    If $k_1 \neq v_2$, $k_2 \neq v_3$ and $k_3 \neq v_1$, then $k_1,v_1,k_2,v_2,k_3,v_3$ is a cycle of length six in $\Gamma$, therefore giving us a contradiction. In any other case, for example if $k_1 = v_2$, then $(v_1,v_2) \in E_{\Gamma}$, and then $v_1,v_2,k_2$ forms a triangle in $\Gamma$ once again giving us contradiction. Therefore every clique with at least two vertices is monochromatic. 
    \newline
    Observe now, that there are no isolated vertices, or all vertices all isolated, because there are no vertices $u,v \in V$ such that $N(u) \subsetneq N(v)$. If all vertices are isolated, then since $\Gamma$ is reduced, it contains only one vertex, and thesis is obviously true. In the second case, let $\pi_0 \in \text{Aut}(\Gamma^2)$ and $v$ be any vertex of $\Gamma$. Then $N_{\Gamma}(v)$ is a clique in $\Gamma^2$. It is maximal, because if there was a clique $K_0$ such that $N_{\Gamma}(v) \subsetneq K_0$, then, since all cliques in $\Gamma^2$ are monochromatic, there would exist $w \in V$ such that $N_{\Gamma}(v) \subsetneq N_{\Gamma}(w)$. This obviously contradicts our assumptions. Since $\pi_0$ is an automorphism of $\Gamma^2$, every maximal clique is mapped to a maximal clique, therefore $\pi(N_{\Gamma}(v))$ forms a maximal clique in $\Gamma^2$. If $\pi(N_{\Gamma}(v))$ have at least two elements, then since it is monochrmoatic and maximal, it forms a full neighbourhood of some vertex in $\Gamma$. If on the other hand, $\pi(N_{\Gamma}(v)) = \{u\}$, then there exists some vertex $u' \in V$ such that $(u,u') \in E_{\Gamma}$, because $\Gamma$ does not have isolated vertices. In this case we get $N_{\Gamma}(u') = \{u\}$, otherwise $\{u\}$ would not be a maximal clique in $\Gamma^2$.
    \newline
    We have proven, that every neighbourhood in $\Gamma$ is mapped to some neighbourhood by $\pi_0$, and therefore $\pi_0 \in \text{Aut}^{\pi}(\Gamma)$. This proves that $\text{Aut}(\Gamma^2) \subseteq \text{Aut}^{\pi}(\Gamma)$. Theorem \ref{THM: Aut^pi(Gamma) < Aut(Gamma^2)} ensures us, that the converse holds. 
\end{proof}

    The above theorem may be used to determine whenever some Cayley graphs are stable. Since they do not contain vertices $u,v$ such that $N(u) \subsetneq N(v)$, to apply Theorem \ref{THM: brak cykli 3,6 i podwierzcholkow implikuje ze Aut^pi(Gamma) = Aut(Gamma^2)}, one only have to check whenever it contains triangles or cycles of length six. If no, then $\text{Aut}^{\pi}(\text{Cay}(G,S)) \cong \text{Aut}(\text{Cay}(G,S^2 \backslash \{e\}))$.

\section{Topological Automorphisms} \label{SECTION: Properties of Topological Automorphisms}

In this section we will present two results that may provide an useful tool in proving stability for some families of graphs. First let us rewind one definition.



\begin{defi}
    Let $\Gamma = (V,E)$ be a finite simple graph and $v$ be a vertex.\textit{ A ball of radius} $r$ \textit{centered at} $v$, denoted by $B(v,r)$ is a set 
    $$
    B(v,r) = \{w \in V \mid d(v,w) \leq r \}.
    $$
\end{defi}

\begin{thm} \label{TW: Aut^Tau zachowują kule dowolnych promieni}
    Let $\Gamma = (V,E)$ be a graph such that ${\Gamma}^C$ is reduced. Let $\pi_0 \in \text{Aut}^{\tau}(\Gamma)$. Then
    $$
    \forall_{r \in {\mathbb{Z}}_{>0} } \quad \forall_{v \in V} \quad \pi_0(B(v,r)) = B({\gamma}^{r}(\pi_0) (v),r ),
    $$
    where $\gamma$ is the function, (cf. Definition \ref{definicja funkcji gamma}) defined for $\text{Aut}^{\pi}({\Gamma}^C)$.
\end{thm}
\begin{proof}
    Let's start the proof by observing
    $$
    B(v,2) = \bigcup_{u \in \mathcal{U}} B(u,1),
    $$
    where $\mathcal{U}$ is the set of all vertices $u$ such that $v \in B(u,1)$. Therefore 
    $$
    \pi_0(B(v,2)) = \bigcup_{u \in \mathcal{U}'} B(u,1) = B(\pi_0(v),2),
    $$
    where $\mathcal{U}'$ is the set of all vertices $u'$ such that $\pi_0(v) \in B(u',1)$. This proves our thesis for $r = 1,2$.
    \newline

    We now proceed by induction on $r$. Similarly as for $r=2$ let us start with the following observation:
    $$
    B(v,r) = \bigcup_{u \in {\mathcal{U}}_r } B(u,1),
    $$
    where ${\mathcal{U}}_r$ is the set of all vertices $u$ such that the intersection $B(u,1) \cap B(v,r-2)$ is nonempty.
    To finish the proof it is enough to see that
    $$
    \pi_0(B(v,r)) = \bigcup_{u \in \mathcal{U}'_r } B(u,1) = B(\gamma^{r}(\pi_0)(v),r),
    $$
    where $\mathcal{U}'_r$ is the set of all vertices $u'$ such that the intersection $B(u,1) \cap B(\gamma^{r}(\pi_0)(v),r-2)$ is nonempty. This follows, since $\pi_0(B(v,r-2)) = B(\gamma^{r-2}(\pi_0)(v),r-2)$, and \\ $B(\gamma^{r-2}(\pi_0)(v),r-2) = B(\gamma^{r}(\pi_0)(v),r-2)$ because $\gamma^2 \equiv Id$.
\end{proof}

\begin{cor}
    Let $\Gamma = (V,E)$ be a graph such that ${\Gamma}^C$ is reduced. Let $\pi_0 \in \text{Aut}^{\tau}(\Gamma)$. Take $v,w \in V$ to be any pair of vertices. Then for every $n \in {\mathbb{Z}}_{>0}$ the following equivalence holds
    $$
    d(v,w) \in \{ 2n-1,2n \} \quad \Longleftrightarrow \quad d(\pi_0(v),\pi_0(w)) \in \{ 2n-1,2n \}.
    $$
\end{cor}
\begin{proof}
    Let firstly notice that $d(v,w) \leq 2n$ if and only if there exists such $x \in V$ that $d(v,x),d(x,w) \leq n$, and that is equivalent to the fact that $v,w \in B(x,n)$. By Theorem \ref{TW: Aut^Tau zachowują kule dowolnych promieni} there exists such $x'$ that $\pi_0(v),\pi_0(w) \in B(x',n)$. By the same logic as above, it is equivalent to the fact that $d(\pi_0(v),\pi_0(w)) \leq 2n$, which completes the proof.  
\end{proof}

\section{Construction of conected nonbipartite (H,$\sigma$)-graphs} \label{SECTION: Construction of (H,sigma) conected nonbipartite graphs}

This section is focused on constructing connected, reduced and nonbipartite graphs with the two-fold automorphisms group isomorphic to $\text{H} \rtimes \mathbb{Z}_2$ for any prescribed $\text{H}$.

\begin{defi}
    ($\text{H}$,$\sigma$)-graph, where $\text{H}$ is a group and $\sigma \in \text{Aut}(\text{H})$ is a reduced graph $\Gamma$ such that there exist isomorphism $\psi:{\text{Aut}}^{\pi}(\Gamma) \rightarrow \text{H}$, for which $ \sigma = \psi \circ \gamma \circ \psi^{-1}$ holds.
\end{defi}

Later in this section we construct a graph $\Gamma_{(\text{H},\sigma)}$, and then we prove that it satisfied the aforementioned criteria. To prove that it indeed does, we firstly need to prove a couple lemmas. 

\begin{defi}
     Let $\Gamma = (V,E)$ be a reduced graph, and let $\mathcal{O}_{\text{Aut}^{\pi}(\Gamma)}$ be a partition of $V$ into orbits of action of a group $\text{Aut}^{\pi}(\Gamma)$ on $V$.
    \newline
    Let $\mathcal{P}_1$ and $\mathcal{P}_2$ be two partitions of $V$. We say, that $\mathcal{P}_1$ is a prepartition of $\mathcal{P}_2$, or that $\mathcal{P}_2$ is a subpartition of $\mathcal{P}_1$ if and only if
    $$
    \forall_{P \in \mathcal{P}_2} \quad \exists_{Q \in \mathcal{P}_1} \quad P \subseteq Q.
    $$
\end{defi}

Following lemma will play the crucial role proving that a group isomorphic to $\text{H}$ is not only a subgroup of $\text{Aut}^{\pi}(\Gamma_{(\text{H},\sigma)})$, but is actually equal to this group.

\begin{lem} \label{lemat o trojkącie z ,,blobow'' w konstrukcji}
    Let $\text{H}$ be a finite group, $\sigma \in \text{Aut}(\text{H})$ be such that $\sigma^2 \equiv Id$, $\Gamma = (V,E)$ be a reduced graph, and let $\pi_0 \in \text{Aut}^{\pi}(\Gamma)$. Take $\mathcal{P} = \{V_1,V_2,V_3,V_4\}$ to be a partition of $V$ which is a prepartition of $\mathcal{O}_{\text{Aut}^{\pi}(\Gamma)}$. If there exists such labelings $\psi_i$ of vertices $V_i$, for $i=1,2,3$, $|V_i|=|\text{H}|$ for $i=1,2,3$ and for all $u \in V_1$, $v \in V_2$ and $w \in V_3$
    \begin{align*}
         (u,v) \in E \Longleftrightarrow \sigma(\psi_1(u)) = \psi_2(v) \quad \text{and} \quad (v,w) \in E \Longleftrightarrow \sigma(\psi_2(v)) = \psi_3(w), 
    \end{align*}
    then permutations of $V_1$ and $V_3$ induced by $\pi_0$ with respect to labelings $\psi_1$ and $\psi_3$ give the same permutation of $\text{H}$.
\end{lem}
\begin{proof}
     Take $u_1,v_1 \in V_1$ such that $\pi_0(u_1) = v_1$, $\psi_1(u_1) = x$ and $\psi_1(v_1) = y$. Let $u_2,v_2 \in V_2$ be such that $\psi_2(u_2) = \sigma(x)$ and $\psi_2(v_2) = \sigma(y)$. Consider now $\gamma(\pi_0)$, more precisely the image of $u_2$. Since $N(u_2) \cap V_1 = \{u_1\}$, we get $\pi_0(N(u_2)) \cap V_1 = \{v_1\}$. Since $v_1 \in \pi_0(N(u_2)) = N(\gamma(\pi_0)(u_2))$ and $\gamma(\pi_0)(u_2) \in V_2$ we conclude that $\gamma(\pi_0)(u_2) = v_2$.
    \newline
    Let us now define $u_3,v_3 \in V_3$ such that $\psi_3(u_3) = x$ and $\psi_3(v_3) = y$. Similarly as above, we conclude that $\pi_0(u_3) = v_3$, from the fact that $\gamma(\pi_0)(u_2) = v_2$ and $\sigma^2 \equiv Id$. This finishes the proof.
    \center
    \includegraphics[scale = 1]{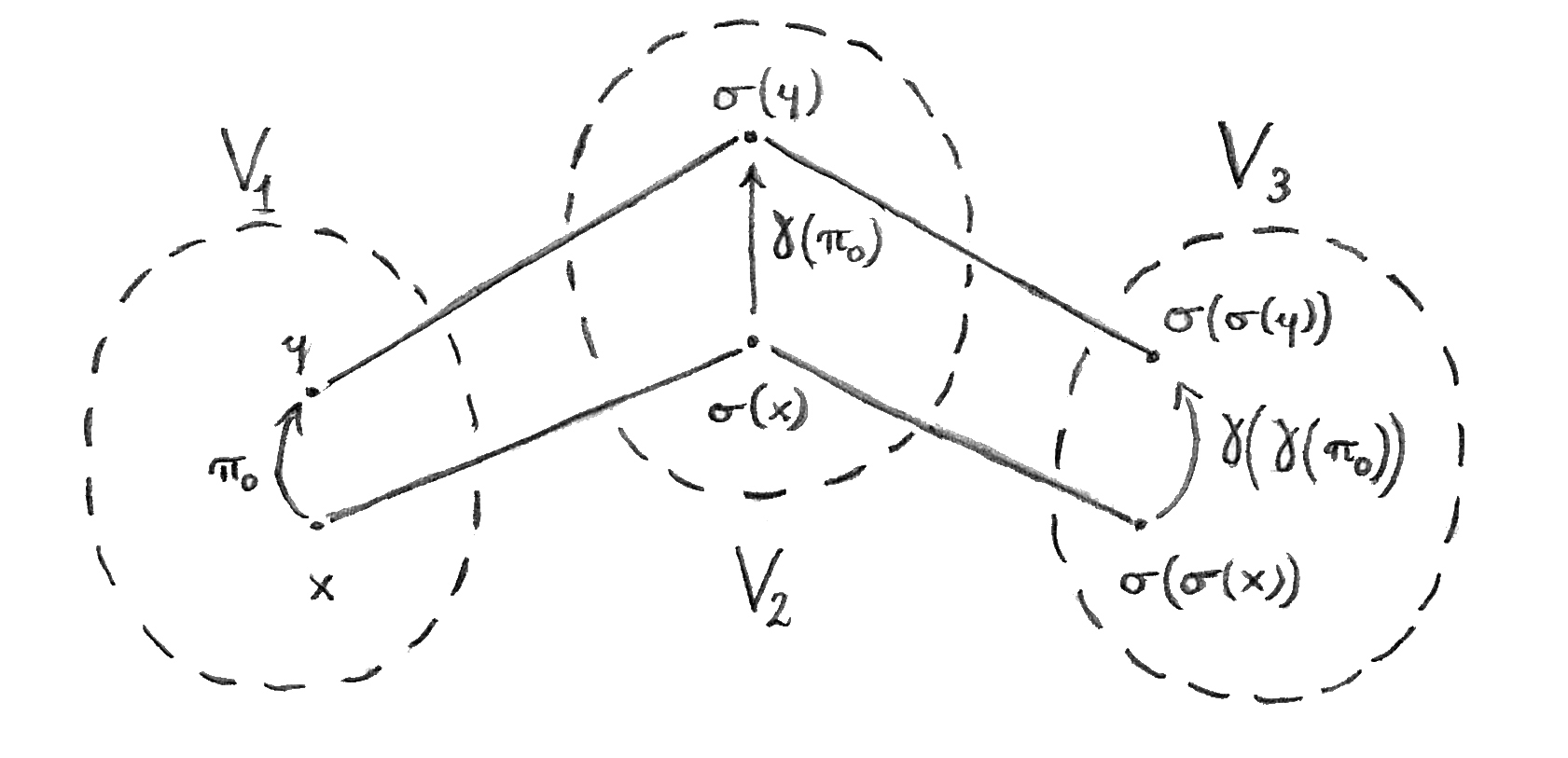}
    
\end{proof}

Following lemma easily produces invariants of the action $\text{Aut}^{\pi}(\Gamma_{(\text{H},\sigma)})$ on vertices. It will play an important part in determining orbits of the action $\text{Aut}^{\pi}(\Gamma_{(\text{H},\sigma)})$. For stating this lemma we need to give some definitions first.

\begin{defi}
    Let $\text{H}$ be a group, $X,Y$ any sets and $\phi: \text{H} \rightarrow \text{Sym}(X)$ be a group homomorphism. We call a function $f: X \rightarrow Y$ an invariant of a group $\text{H}$ with respect to $\phi$ if for any $h \in \text{H}$ and any $x \in X$, $f(\phi(h)(x)) = f(x)$. When $X$ and $\phi$ are clear from the context we call $f$ and invariant of $\text{H}$.
\end{defi}

\begin{defi}
    \textit{Multiset} of a set $X$ is any function $X_{\text{multi}}: X \mapsto \mathbb{Z}_{\geq 0}$. We call elements $x \in X$ such that $X_{\text{multi}}(x) > 0$ \textit{elements of multiset} $X_{\text{multi}}$. For any element $x$ of the multiset $X_{\text{multi}}$ we will refer to the number $X_{\text{multi}}(x)$ as \textit{multiplicity of element} $x$. 
\end{defi}

\begin{defi}
    \textit{Multi power set} of a set $X$ is a set that contains all multisets of a set $X$. It is denoted with ${\mathcal{P}(X)}_{\text{multi}}$.
\end{defi}

\begin{lem} \label{Lemat: Niezmienniki Aut^pi}
Let $\Gamma$ be a reduced graph, $X$ a set and let $f:V \rightarrow X$ be an invariant of $\text{Aut}^{\pi}(\Gamma)$. Then the function $\text{Ne}(f): V \rightarrow {\mathcal{P}(X)}_{\text{multi}}$ defined by
$$
\text{Ne}(f): \quad v \text{ } \mapsto \text{ } \{ f(w) \mid w \in N(v) \}_{\text{multi}}
$$
also is an invariant of $\text{Aut}^{\pi}(\Gamma)$. 
\end{lem}
\begin{proof}
    For $\pi_0 \in \text{Aut}^{\pi}(\Gamma)$ we have
    $$
    \text{Ne}(f)(\pi_0(v)) = \{ f(w) \mid w \in N(\pi_0(v)) \}_{\text{multi}} = \{ f(w) \mid w \in \gamma(\pi_0)(N(v)) \}_{\text{multi}}.
    $$
    Since $f$ is an invariant under every element of $\text{Aut}^{\pi}(\Gamma)$, this applies also to $\gamma(\pi_0)$, therefore
    $\text{Ne}(f)(\pi_0(v)) = \{ f(w) \mid w \in N(v) \}_{\text{multi}} = \text{Ne}(f)(v)$.
\end{proof}

Now we will define a graph, which will be a ,,skeleton'' for our construction.

\begin{defi}
    Let $n_0 \geq 6$ be a positive integer. By $R(n_0)$ we understand a graph with vertex set $V_{R(n_0)} = \{1, \ldots, n_0 + 6\}$. Now let us define $$
    E_1 = \{(i,i+1) \mid 1\leq i \leq n_0 - 1 \} \cup \{(2,4)\}, \quad E_2 = \{(i,n_0 + 1),(i,n_0 + 2)\mid 1\leq i \leq n_0 - 1 \}
    $$
    $$
    \text{and} \quad E_3 = \{(n_0+i,n_0+i+1) \mid 1\leq i \leq 5 \} \cup \{(n_0+2,n_0+4),(n_0+2,n_0+6),(n_0+4,n_0+6)\}.
    $$
    Then $E_{R(n_0)} = E_1 \cup E_2 \cup E_3$.
    \newline
    \center
    \includegraphics[scale = 1]{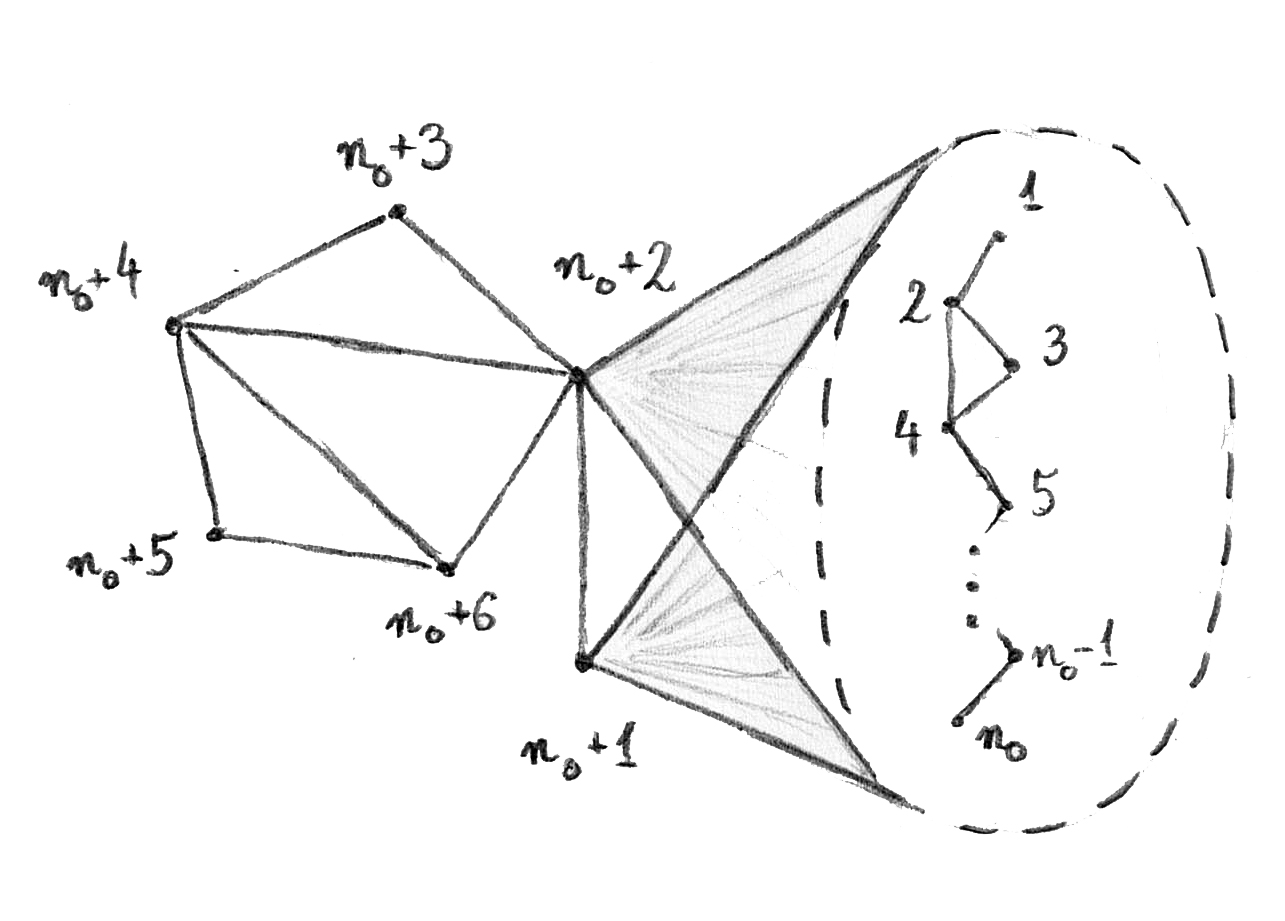}
    
\end{defi}

Let us now prove lemma which will become useful in the proof of Theorem \ref{konstrukcja (H,sigma) grafow}.

\begin{lem} \label{LEM: niezmienniki R(n_0) determinujące wierzcholki}
    For any $n_0 \geq 6$ and there exists $k \in \mathbb{Z}_{\geq 0}$ such that for any vertices $u,v$ of the graph $R(n_0)$ 
    $$
    \text{Ne}^k(\text{deg})(u) \neq \text{Ne}^k(\text{deg})(v).
    $$
\end{lem}
\begin{proof}
    For the use of this proof we call a vertex $v$ ,,determined'', if there exists $k_0$ such that the value of $\text{Ne}^{k_0}(\text{deg})(v)$ among all elements of $V$. We will show that all vertices are determined.
    \indent We will start by recognizing that degrees of $n_0+3$ and $n_0+5$ are equal to two, degrees of vertices $1,n_0,n_0+6$ are all equal to three, degrees of $5,\ldots,n_0-1$ and $3,n_0+4$ are equal to five, vertices $2,3,4$ have degree six. For remaining vertices we have $\text{deg}(n_0+1) = n_0 +1$ and $\text{deg}(n_0+2) = n_0 +4$.
    \newline
    Vertices $n_0+1$ and $n_0+2$ have unique degrees and therefore thEY are determined. Observe that if a vertex $v \neq n_0 + 1$ which have a number $n_0+4$ in the multiset $\text{Ne}(\text{deg})(v)$, then it is an element of $V_1 = \{1,\ldots n_0\}$. Notice now that only vertices contained in $V_1$ with degree three are $1$ and $n_0$, however $1$ is connected to a vertex of degree five and vertex $n_0$ is not.
    \newline
    Since vertex $1$ is determined, vertex $2$ is determined too, as it is the only vertex from $V_1$ connected to $1$. Now $3$ is also determined, as the only vertex in $V_1$ of degree four connected to vertex $2$, and therefore $4$ is determined as the only vertex from $V_1$, other than $3$ and connected to $2$. Now we proceed inductively to prove that all vertices in $V_1$ are determined.
    \newline
    Now, since we had already shown that $n_0+1$ and $n_0+2$ are determined, vertices $n_0+3$, $n_0+4$ and $n_0+6$ are also determined, since those are the only neighbours of vertex $n_0+2$ which do not belong to $V_1$, they are different from already determined vertex $n_0+1$ and all of them have different degrees. The only vertex left is $n_0+5$ and therefore it is determined as well, as there exists such $k_0$, that each vertex $v \in V\text{ } \backslash\text{ } \{n_0+5\}$ have unique multiset $\text{Ne}^{k_0}(\text{deg})(v)$. Therefore $\text{Ne}^{k_0}(\text{deg})(n_0+5)$ also is unique, as it cannot be equal to $\text{Ne}^{k_0}(\text{deg})(v)$ for any $v \in V\text{ } \backslash\text{ } \{n_0+5\}$.
\end{proof}

One can be interested in constructing graphs with prescribed two-fold automorphism group with as few vertices as possible. For that reason in our construction we will use some minimal generating set of $\text{H}$. Notice that such set have at most $\text{log}_2(|\text{H}|)$ elements.

\begin{defi}
    For any finite group $\text{H}$ by $\text{rank}(\text{H})$ we denote the minimal cardinality of a set $X \subseteq \text{H}$ satisfying $ \left\langle X \right\rangle  = \text{H}$.
\end{defi}

Now we are ready to define graph $\Gamma_{(\text{H},\sigma)}$ which will turn out to be connected, reduced, nonbipartite and with desired two-fold automorphism group.

\begin{defi} \label{DEFI: Gamma_{H,sigma, X}}
    Let $\text{H}$ be a finite group and $\sigma \in \text{Aut}(\text{H})$ be such that $\sigma^2 \equiv Id$. Let $X = \{x_1, \ldots x_{\text{rank}(\text{H})}\}$ be a generating set of $\text{H}$. Let $n_0 = 1 + \text{max}(\text{rank}(\text{H}),5)$.
    Define
    $$
    E_1 = \{((h_1,i),(h_2,j)) \mid \sigma(h_1) = h_2 \text{ and } i,j \neq n_0+1 \text{ and } (i,j) \in E_{R(n_0)}\},
    $$
    $$
    E_2 = \{((h_1,i),(h_2,n_0+1)) \mid \sigma(h_1) \cdot x_i = h_2 \text{ and } i \leq \text{rank}(\text{H}) \text{ and } (i,n_0+1) \in E_{R(n_0)} \},
    $$
    $$
    E_3 = \{((h_1,i),(h_2,n_0+1)) \mid \sigma(h_1) = h_2 \text{ and } i > \text{rank}(\text{H}) \text{ and } (i,n_0+1) \in E_{R(n_0)}\}.
    $$
    We define the graph $\Gamma_{(\text{H},\sigma),X} = (V,E)$ by putting $V = \text{H} \times \{1,\ldots, n_0 + 6\}$ and $E = E_1 \cup E_2 \cup E_3$. 
\end{defi}

We will start by showing, that $\Gamma_{(\text{H},\sigma),X}$ is indeed connected, reduced and nonbipartite graph, as those properties are easier to show, and will be helpful in proving the missing result about two-fold automorphisms group.

\begin{lem}
     Let $\text{H}$ be a finite group and $\sigma \in \text{Aut}(\text{H})$ be such that $\sigma^2 \equiv Id$. Let $X = \{x_1, \ldots x_{\text{rank}(\text{H})}\}$ be a generating set of $\text{H}$. Then $\Gamma_{(\text{H},\sigma),X}$ is reduced, connected and nonbipartite graph.
\end{lem}
\begin{proof}
    Let us first prove that $\Gamma_{(\text{H},\sigma),X}$ is reduced. Assume that is not the case. Note that a vertex $(h,i)$ have edges to vertices $(h',j)$ where $j$ is such that $(i,j) \in E_{R(n_0)}$. It follows from Lemma \ref{LEM: niezmienniki R(n_0) determinujące wierzcholki} that $R(n_0)$ is reduced, so if two vertices $(h_1,i)$ and $(h_2,j)$ from $\Gamma_{(\text{H},\sigma),X}$ have the same neighbourhoods, then $i=j$. Since for any $i,j \in \mathbb{Z}$ and $h \in \text{H}$ vertex $(h,i)$ is connected by an edge to at most one vertex from $\text{H}\times \{j\}$ and it is determined by $h$ we conclude that $h_1 = h_2$ which ends the proof that $\Gamma_{(\text{H},\sigma),X}$ is reduced.
    \newline
    Now, let us prove that $\Gamma_{(\text{H},\sigma),X}$ is connected. We start by showing, that all vertices of $\text{H} \times \{n_0 + 2\}$ are in the same connected component. To achieve this we first prove that for all $1\leq i \leq \text{rank}(\text{H})$ and $h \in \text{H}$, vertices $(h,n_0+2)$ and $(h\cdot x_i, n_0 +2)$ are in the same connected component.
    \newline
    Consider a path $(h,n_0+2),(\sigma(h),i),(h \cdot x_i,n_0+1),(\sigma(h \cdot x_i),n_0),(h \cdot x_i,n_0+2)$. Now, the claim that $\text{H} \times \{n_0 + 2\}$ is contained in asingle connected component of $\Gamma_{(\text{H},\sigma),X}$ follows from the fact that $X$ is a generating set of $\text{H}$. Connectedness of the graph $R(n_0)$ ends the proof of connectedness.
    \newline
    To finish the proof we have to show that $\Gamma_{(\text{H},\sigma),X}$ is nonbipartite, however it  contains a triangle $(e,n_0+2),(e,n_0+3),(e,n_0+4)$, and the proof is done.
\end{proof}

We are now ready to prove the main theorem of this section.

\begin{thm} \label{konstrukcja (H,sigma) grafow}
     Let $\text{H}$ be a finite group and $\sigma \in \text{Aut}(\text{H})$ be such that $\sigma^2 \equiv Id$. Then there exists $(\text{H},\sigma)$-graph $\Gamma_{(\text{H},\sigma)} = (V,E)$ which is connected, nonbipartite, and have exactly $|\text{H}| \cdot (7 + \text{max}(\text{rank}(\text{H}),5) )$ vertices.
     Moreover, all orbits of ${\text{Aut}}^{\pi}(\Gamma_{(\text{H},\sigma)})$ consists of exactly $|\text{H}|$ vertices and graphs induced by those orbits are cocliques.
\end{thm}
\begin{proof}
    Take any generating set $X$ of cardinality $\text{rank}(\text{H})$. We will prove that $\Gamma_{(\text{H},\sigma)} = \Gamma_{(\text{H},\sigma),X}$ satisfies all the conditions stated in the theorem. Since permutations from $\text{Aut}^{\pi}(\Gamma_{(\text{H},\sigma)})$ correspond to automorphisms of $\text{B}\Gamma_{(\text{H},\sigma)}$, the degree is an invariant of $\text{Aut}^{\pi}(\Gamma_{(\text{H},\sigma)})$. Therefore, by Lemma \ref{LEM: niezmienniki R(n_0) determinujące wierzcholki}, for any $k \in \mathbb{Z}$ function $\text{Ne}^k(\text{deg})$ is an invariant of $\text{Aut}^{\pi}(\Gamma_{(\text{H},\sigma)})$. Hence $\{ \text{H} \times \{i\} \mid 1 \leq i \leq n_0 + 6\}$ is a prepartition of $\mathcal{O}_{\text{Aut}^{\pi}(\Gamma_{(\text{H},\sigma)})}$. On the other hand, let $h$ be any element of $\text{H}$. Then the function $h\cdot$ defined by formula $h\cdot : (h_0,i) \mapsto (h\cdot h_0,i)$ is a two-fold projection, because $(h\cdot,\sigma(h)\cdot) \in \text{Aut}^{\text{TF}}(\Gamma_{(\text{H},\sigma)})$. This shows that in fact $\mathcal{O}_{\text{Aut}^{\pi}(\Gamma_{(\text{H},\sigma)})} = \{ \text{H} \times \{i\} \mid 1 \leq i \leq n_0 + 6\}$, and therefore proves that subgraphs induced by those are in fact cocliques. It also implies that $\text{Aut}^{\pi}(\Gamma_{(\text{H},\sigma)})$ contains a subgroup isomorphic to $\text{H}$. We will now focus on proving, that this copy of $\text{H}$ is the full group of two-fold projections of $\Gamma_{(\text{H},\sigma)}$.
    \newline
    Let $\pi_0 \in \text{Aut}^{\pi}(\Gamma_{(\text{H},\sigma)})$. We want to show that there exists $h_0 \in \text{H}$ such that $\pi_0 \equiv h_0 \cdot$. We start by proving, that for each $i \in \{1,\dots, \text{rank}(\text{H})\}$, if $\pi_0((g,n_0+2)) = (h \cdot g,n_0+2)$ then $\pi_0((g\cdot x_i,n_0+2)) = (h \cdot g \cdot x_i,n_0+2)$.
    \newline
    By applying Lemma \ref{lemat o trojkącie z ,,blobow'' w konstrukcji} with $V_1 = \text{H} \times \{n_0 +2 \}$, $V_2 = \text{H} \times \{i\}$, $V_3 = \text{H} \times \{n_0 +1 \}$ and functions $\psi_i$ given by formulas
    $$
    \psi_1 : (h,n_0+2) \mapsto h, \quad \psi_2 : (h,i) \mapsto h, \quad \text{and} \quad \psi_3 : (h \cdot x_i,n_0+1) \mapsto h,
    $$ 
    shows, that $\pi_0((g,n_0+1)) = (h \cdot g,n_0+1)$. Now applying Lemma \ref{lemat o trojkącie z ,,blobow'' w konstrukcji} once again with $V_1 = \text{H} \times \{n_0 + 1 \}$, $V_2 = \text{H} \times \{n_0\}$, $V_3 = \text{H} \times \{n_0 +2 \}$ and functions $\psi_i$ given by projections to the first factor proves our claim.
    \newline
    Since above reasoning can be applied for any element of a generating set $X$, it shows that there exists $h_0 \in \text{H}$ such that for every $h \in \text{H}$ it is true that $\pi_0((h,n_0+2)) = (h_0 \cdot h,n_0+2)$. To finish the proof, we have to show that $\pi_0$ acts in the same way on other vertices of $\Gamma_{(\text{H},\sigma)}$.
    \newline
    We will once again apply Lemma \ref{lemat o trojkącie z ,,blobow'' w konstrukcji}. Put $V_1 = \text{H} \times \{n_0 + 2 \}$, $V_2 = \text{H} \times \{n_0\}$, $V_3 = \text{H} \times \{n_0 +1 \}$ and define functions $\psi_i$ as projections to the first coordinate we obtain that $\pi_0$ acts the same as $h_0\cdot$ not only on $\text{H} \times \{n_0 + 2 \}$ but also on $\text{H} \times \{n_0 + 1 \}$.
    \newline
    One can observe, that if $1 \leq p,q,r \leq n_0 +6$ are all different from $n_0 + 1$ and both $(p,q)$, $(q,r)$ are edges in $E_{R(n_0)}$, then applying Lemma \ref{lemat o trojkącie z ,,blobow'' w konstrukcji} for $V_1 = \text{H} \times \{p\}$, $V_2 = \text{H} \times \{q\}$, $V_3 = \text{H} \times \{r\}$ and functions $\psi_i$ given by projection to the first coordinate we obtain that if $\pi_0$ acted the same as $h_0 \cdot$ on $\text{H} \times \{p\}$, it also does on $\text{H} \times \{r\}$. Since the graph $R(n_0)$ is connected and nonbipartite after removing vertex $n_0 + 1$, one can find a path of even length between any vertices of this graph. Reasoning above with this observation about $R(n_0)$ shows, that in fact $\pi_0 \equiv h_0 \cdot$, which ends the proof of the theorem.
\end{proof}

\section{Description of $\alpha$-automorphic graphs} \label{SECTION: Clasification and construction of alpha-automorphic graphs}

Now we will investigate a rare case of asymmetric unstable graphs.

\begin{defi}
    We call a graph $\Gamma$ an $\alpha$\textit{-automorphic graph} if and only if $\text{Aut}(\Gamma) = \{ Id \}$ and ${\text{Aut}}^{\pi}(\Gamma) \neq \{ Id \}$.
\end{defi}

$\alpha$-automorphic graphs were at first recognized in \cite{zbMATH05249593} by Wilson, and then further studied in \cite{zbMATH07349648} by Lauri, Mizzi and Scapellato. In \cite{zbMATH07349648} they showed how to construct an $\alpha$-automorphic graph $\Gamma$ with $4m$ vertices, such that ${\text{Aut}}^{\pi}(\Gamma) \cong {\mathbb{Z}}_{m}$ for any odd integer $m$. They noticed the importance of the set $\text{Ant}(\Gamma)$, however they did not recognize, that ${\text{Aut}}^{\pi}(\Gamma) = \text{Ant}(\Gamma)$ if given graph is asymmetric. Following theorem proves it in detail. For any group $\text{H}$ we denote the function that maps every element to its' inverse by ${ [\cdot] }^{-1}$.

\begin{thm} \label{THM: charakteryzacja alpha-automorficznych}
    If $\Gamma = (V,E)$ is an $\alpha$-automorphic graph, then ${\text{Aut}}^{\pi}(\Gamma)$ is an abelian group of odd order, and $\gamma \equiv { [\cdot] }^{-1}$. Moreover, supgraph induced by any orbit of ${\text{Aut}}^{\pi}(\Gamma)$ on $V$ is an empty graph.
\end{thm}
\begin{proof}
    At first we can notice that $\Gamma$ is reduced. If not, there would exist $u,v \in V$ with $N(u) = N(v)$. Then the transposition of those two would be a nontrivial automorphism. Let us now apply Corollary \ref{COR: parzystosc Aut taka sama jak parzystosc Aut^pi} with $|\text{Aut}(\Gamma)| = 1$. We get that ${\text{Aut}}^{\pi}(\Gamma)$ is a group of odd order. By Theorem \ref{THM: alpha jest stala na prawostronnych warstwach Aut} we obtain that ${\text{Aut}}^{\pi}(\Gamma) = \text{Im}(\alpha)$, and by applying first part of Observation \ref{OBS wlasnosci alphy i elementow odwracanych przez gamme} we get that $\gamma \equiv { [\cdot] }^{-1}$. Since $\gamma$ is an automorphism of ${\text{Aut}}^{\pi}(\Gamma)$, it follows that it is an abelian group.
    \newline
    Now, we prove, that there are no edges between vertices in the same orbit, however it is a straightforward application of Corollary \ref{COR: brak krawedzi w orbitach alphy} and the fact that ${\text{Aut}}^{\pi}(\Gamma) = \text{Im}(\alpha)$. 
\end{proof}

 As mentioned earlier, Lauri, Mizzi and Scapellato showed, that any cyclic group of odd order can be isomorphic to a group of TF-projections of $\alpha$-automorphic graph. Result below shows, that in fact limitations showed in Theorem \ref{THM: charakteryzacja alpha-automorficznych} are the only existing ones.

\begin{obs}
    If $\text{H}$ is nontrivial abelian group of odd order, then there exist $(\text{H},{ [\cdot] }^{-1})$-graph which is $\alpha$-automorphic.
\end{obs}
\begin{proof}
    It is enough to apply Theorem \ref{konstrukcja (H,sigma) grafow}. Taking the inverse twice yields identity, and $x^{-1} = \gamma(x) = x$ implies $x^2 = e$. Since $\text{H}$ has an odd order, it implies that $x = e$. Therefore the group ${\text{Aut}}(\Gamma_{(\text{H},{ [\cdot] }^{-1})})$ is in fact trivial.
\end{proof}

This ends the description of $\alpha$-automorphic graphs. We have now necessary tools to prove the following theorem about rarity of graphs with nontrivial TF-automorphism group among all graphs with a given number of vertices.

\begin{thm}
    Let $n \in \mathbb{Z}_{>0}$ be a positive integer, and $\Gamma$ be a graph randomly chosen from a set of $2^{n \choose 2}$ graphs with $n$ vertices. Assume that each graph can be chosen with the same probability. Let $\mathcal{P}(n)$ denote the probability that a graph of $\Gamma$ with $n$ vertices have nontrivial $\text{Aut}^{\pi}(\Gamma_0)$ group. Then
    $$
    \lim_{n \to \infty} \mathcal{P}(n) = 0.
    $$
\end{thm}
\begin{proof}
    In \cite[Theorem 2]{zbMATH03192673} Erdos and Renyi, proved that for every $\varepsilon > 0$
    $$
    \lim_{n \to \infty} \quad \sum_{q = 2}^{n} \text{ } \frac{A_{n,q} \cdot B_{n,q} \cdot  C_{n,q} }{2^{n \choose 2}} = 0,
    $$
    where $A_{n,q}$ is the number of permutations of $n$ vertices which have exactly $n-q$ fixed points and $B_{n,q}$ is the upper bound of the number of graphs $\Gamma_0$ which are invariant under chosen permutation of $n$ vertices with $n-q$ fixed points. $C_{n,q}$ denotes an upper bound on number of graphs ${\Gamma_0}'$ which can be transformed into $\Gamma_0$ by changing $ m \leq \frac{n}{2} \cdot (1- \varepsilon)$ edges into nonedges or vice versa ($m$ is arbitrary, not fixed for $C_{n,q}$), and cannot be transformed into any other graph with nontrivial automorphism group with fewer changes than $m$.
    \newline
    By a single change we understand changing edge into nonedge or nonedge into an edge. We have the following equality
    $$
    \mathcal{P}(n) = \mathcal{P}_A(n) + \mathcal{P}_{\alpha}(n),
    $$
    where $\mathcal{P}_A(n)$ is the probability that a random graph on $n$ vertices have nontrivial automorphism group, and $\mathcal{P}_{\alpha}(n)$ is the probability that a random graph on $n$ vertices is an $\alpha$-automorphic graph. We denote the set of all labeled graphs with $n$ vertices by $\mathcal{G}_n$. In the first case we have
    $$
    \mathcal{P}_A(n) \cdot 2^{n \choose 2} \leq |\{ (\pi,\Gamma) \mid \Gamma \in \mathcal{G}_{n}, \text{ } \pi \neq Id \text{ and } \pi \in \text{Aut}(\Gamma) \}| \leq \sum_{q = 2}^{n} \text{ } A_{n,q} \cdot B_{n,q}.
    $$
    First inequality follows from the fact that any graph with nontrivial automorphism group is counted at least once in the set shown in the middle. Second one follows directly from definitions of $A_{n,q}$ and $B_{n,q}$. Inequality above shows, that 
    $$
    0 \leq \lim_{n \to \infty} \mathcal{P}_A(n) \leq \lim_{n \to \infty} \sum_{q = 2}^{n} \text{ } \frac{A_{n,q} \cdot B_{n,q} \cdot  C_{n,q} }{2^{n \choose 2}} = 0,
    $$
    because $C_{n,q} \geq 1$. It is true due to the fact that if a permutation $\pi_0$ and graph $\Gamma_0$ are fixed, then $\Gamma_0$ is itself one of the graphs counted in $C_{n,q}$. 
    \newline
    We denote the set of all reduced labeled graphs with $n$ vertices by $\mathcal{G}_{\text{red},n}$. For the second case let us consider sets 
    $$
    G_{\alpha} = \{ (\pi_0,\Gamma) \mid \Gamma \in \mathcal{G}_{\text{red},n}, \text{ } \pi_0 \neq Id \text{ and } \pi_0 \in \text{Aut}^{\pi}(\Gamma) \text{ and } \pi_0 \in \text{Im}(\alpha) \}
    $$
    $$
    \text{and} \quad G_{A} = \{ (\pi_0,\Gamma) \mid \Gamma \in \mathcal{G}_{n}, \text{ } \pi_0 \neq Id \text{ and } \pi_0 \in \text{Aut}(\Gamma) \}.
    $$
    We will now construct an injection from $G_{\alpha}$ to $G_{A}$. 
    First, let us pick permutation $\pi_0$ and let $\mathcal{O}$ be the set of orbits. Choose any linear order $\prec$ on $\mathcal{O}$ and pick one element from each orbit. Let $(\pi_0,\Gamma) \in G_{\alpha}$. We will map $(\pi_0,\Gamma)$ to $(\pi_0,\Gamma')$ where $\Gamma'$ is created with following procedure. If $v,w \in O \in \mathcal{O}$, then they are not connected by an edge in $\Gamma$ by Corollary \ref{COR: brak krawedzi w orbitach alphy}, and we do not change that in $\Gamma'$. Let $O_1 \prec O_2$, $u_1 \in O_1$ be an element from $O_1$, and $v \in O_2$. For any $k \in \mathbb{Z}$, we have $({\pi_0}^k(u_1),v) \in E_{\Gamma'}$, if and only if $(u_1,{\pi_0}^{-k}(v)) \in E_{\Gamma}$. Therefore $\pi_0 \in \text{Aut}(\Gamma')$, so $(\pi_0,\Gamma') \in G_{A}$. Our transformation is a bijection, because edges in $\Gamma$ between orbits of $\pi_0$ are determined by edges from one chosen vertex. This due to the fact that $\gamma(\pi_0) = {\pi_0}^{-1}$ and Observation \ref{OBS wlasnosci alphy i elementow odwracanych przez gamme}. Therefore, we get
    $$
    \mathcal{P}_{\alpha}(n) \cdot  2^{n \choose 2} \leq |G_{\alpha}| \leq |G_{A}| \leq \sum_{q = 2}^{n} \text{ } A_{n,q} \cdot B_{n,q}.
    $$
    First inequality comes from the fact that every $\alpha$-automorphic graph have at least one nontrivial two-fold automorphism projection, second one from the injection constructed above, and the third one from definitions of $A_{n,q}$ and $B_{n,q}$. Once again using the fact that $C_{n,q} \geq 1$, we get the following
    $$
    0 \leq \lim_{n \to \infty} \mathcal{P}_{\alpha}(n) \leq \lim_{n \to \infty} \sum_{q = 2}^{n} \text{ } \frac{A_{n,q} \cdot B_{n,q} \cdot  C_{n,q} }{2^{n \choose 2}} = 0.
    $$
    Therefore we get $\lim_{n \to \infty} \mathcal{P}(n) = \lim_{n \to \infty} \mathcal{P}_{A}(n) + \lim_{n \to \infty} \mathcal{P}_{\alpha}(n) = 0+0 = 0$.

\end{proof}

As one can easily observe, this proof is very wasteful, hence we formulate following conjectures.

\begin{conj} \label{CONJ rzadko da sie przeksztalcic graf w taki z nietrywialna Aut^TF}
    Let $\varepsilon > 0$ and $\mathcal{P}_n(\varepsilon)$ denote the probability, that a random labeled graph $\Gamma$ on $n$ vertices can be transformed to a graph $\Gamma'$ with nontivial $\text{Aut}^{\pi}(\Gamma')$ group by changing at most $\frac{n}{2} \cdot (1 - \varepsilon)$ edges and nonedges. Then
    $$
    \lim_{n \to \infty} \mathcal{P}_n(\varepsilon) = 0.
    $$
\end{conj}

\begin{conj} \label{CONJ: wartosc oczekiwana Aut^pi = 1}
    Let $\Gamma$ be a random labeled graphs on $n$ vertices. Then
    $$
    \lim_{n \to \infty} \mathbb{E}(|\text{Aut}^{\pi}(\Gamma)|) = 1,
    $$
    where by $\mathbb{E}$ we denote the expected value of some function over $\mathcal{G}_n$.
\end{conj}

Proof of Conjecture \ref{CONJ rzadko da sie przeksztalcic graf w taki z nietrywialna Aut^TF} should proceed similar to the proof of \cite[Theorem 2]{zbMATH03192673}, however upper bound on $C_{n,q}$ might be different and have different proof.
\newline
One may also try to prove Conjecture \ref{CONJ: wartosc oczekiwana Aut^pi = 1} by first choosing $\pi_1$ such that $\gamma(\pi_1) = {\pi_1}^{-1}$, automorphism $\pi_2$ and then bound a number of graphs containing both of them.

\section{TF-isomorphisms of graphs via adjacency matrices} \label{SECTION: TF-isomorphisms of graphs via adjacency matrices}

In this Section we will study two-fold isomorphisms between graphs. This notion is equivalent to the fact that two graphs have isomorphic canonical double covers.

\begin{defi}
    We say that graphs $\Gamma_1 = (\text{V}_1,\text{E}_1)$ and $\Gamma_2 = (\text{V}_2,\text{E}_2)$ are \textit{two-fold isomorphic}, if and only if there exists maps $(\pi_1,\pi_2)$ which are bijections from $\text{V}_1$ to $\text{V}_2$ that the following holds
    $$
    \forall_{v,w \in \text{V}_1} \quad (v,w) \in \text{E}_1 \quad \Longleftrightarrow \quad (\pi_1(v),\pi_2(w)) \in \text{E}_2.
    $$
\end{defi}

To change two-fold isomorphism into an isomorphism of double covers one should apply action of $\pi_1$ on one bipartite part and action of $\pi_2$ on the second part.

\begin{defi}
    Let $n \in \mathbb{Z}_{>0}$ be any positive integer, and $q \in \text{Sym}(\{1,\ldots,n\})$ be a permutation of the set $\{1,\ldots,n\}$. Then the \textit{permutation matrix} of $q$ is a $n \times n$ matrix that that for every $i \in \{1,\ldots,n\}$, $i$-th collumn is equal to $e_{q(i)}$, that is the $q(i)$-th element of standard basis. We denote this matrix by $m_q$.
\end{defi}

    Set $\text{Ant}(\Gamma)$ defined below will play an important role in identifying nonisomorphic graphs TF-isomorphic to a given one.
\begin{defi}
    Let $\Gamma$ be a reduced graph. Then by the set of \textit{antymorphisms} we understand
    $$
    \text{Ant}(\Gamma) = \{ \pi_0 \in \text{Aut}^{\pi}(\Gamma) \mid \gamma(\pi_0) = {\pi_0}^{-1} \}.
    $$
\end{defi}
    Zelinka \cite{zbMATH03559606} and Manrusic, Scapellato and Salvi in \cite{zbMATH04108921} proves several results of similar nature to two propositions below.

\begin{defi}
     Let $\Gamma=(V,E)$ be a graph with $n$ vertices and $\psi: V \rightarrow \{1,\ldots,n\}$ be a labeling of vertices. We define an adjacency matrix to be a $n \times n$ 0-1 matrix $A = {(a_{i,j})}_{i,j \leq n}$ such that $a_{\psi(v),\psi(w)} = 1 \Longleftrightarrow (v,w) \in E$.
     When choice of labeling will not matter, \textit{adjacency matrix of a graph} $\Gamma$ is an adjacency matrix with any labeling.
     \newline
     Moreover, for any $V_0 \subseteq V$ we define the \text{characteristic vector}, with respect to labeling $\psi$, of the set $V_0$ to be a 0-1 vector ${(a_i)}_{i \leq n}$ such that $a_{\psi(v)} = 1 \Longleftrightarrow v \in V_0$.  
\end{defi}

\begin{prop}
 \label{PROP: Aut^TF zdefiniowane macierzami}
    Let $\Gamma$ be a reduced graph, and let $A$ be its' adjacency matrix. Then $p \in \text{Sym}(\{1,\ldots,n\})$ is in $\text{Aut}^{\pi}(\Gamma)$ if and only if there exists such $q \in \text{Sym}(\{1,\ldots,n\})$ that
    $$
    m_p A = A m_q.
    $$
    Moreover, if such $q$ exists, then $q = \gamma(p)$.
\end{prop}
\begin{proof}
    To see this let us understand $i$-th column as the characteristic vector of $N(i)$. Then, since $p$ is a permutation of vertices, it permutes elements in those columns. If after that permutation every column became a copy of some other column form $A$, then every column was mapped to a different one since $\Gamma$ is reduced. This permutation of columns is $q$. If such $q$ exists, it permutes columns, which correspond to neighbourhoods and we get $q = \gamma(p)$.
\end{proof}

\begin{prop} \label{PROP: przetlumaczenie TF-izo na macierze i sasiedztwa}
    Reduced graphs $\Gamma_1 = (\text{V}_1,\text{E}_1)$ and $\Gamma_2 = (\text{V}_2,\text{E}_2)$ are TF-isomorphic if and only if there exists bijection $\pi_0$ from $\text{V}_1$ to $\text{V}_2$ such that
    $$
    \forall_{v \in \text{V}_1} \quad \pi_0 (N(v)) \in N(\Gamma_2).
    $$
    
    If $A_1$ and $A_2$ are adjacency matrices of $\Gamma_1$ and $\Gamma_2$ respectively then those are TF-isomorphic if and only if there exists such permutations $p,q \in \text{Sym}(\{1,\ldots,n\})$ that
    $$
    m_{p} A_1  m_{p}^{-1} = A_2 m_{q}^{-1}.
    $$
\end{prop}
\begin{proof}
    For the first part let us consider graphs $\text{B}\Gamma_1$ and $\text{B}\Gamma_2$. It is easy to verify, that $\Gamma_1 {\cong}^{\text{TF}} \Gamma_2$ if and only if there exist a graph isomorphism form $\text{B}\Gamma_1$ to $\text{B}\Gamma_2$ that maps $\text{V}_1 \times \{0\}$ to $\text{V}_2 \times \{0\}$. On the other hand, if we understand vertices from $\text{V}_1 \times \{1\}$ and $\text{V}_2 \times \{1\}$ as neighbourhoods of $V_1$ and $V_2$ respectively, then equivalence to existance of $\pi_0$ becomes obvious, since both $\Gamma_1$ and $\Gamma_2$ are reduced.
    \newline
    Suppose $A_1$ and $A_2$ are coresponding adjacency matrices according to labelings $\psi_1$ and $\psi_2$. Let then $p$ be a permutation of $\{1,\ldots,n\}$ that matches $\pi_0$ after labelings $\psi_1$ and $\psi_2$. If we change the numeration of $V_1$ by applying $p$ we get that adjacency matrix of $\Gamma_1$ with those labels equals $m_{p} A_1  m_{p}^{-1}$. Then map $\pi_0$ from $V_1$ to $V_2$ induces identity permutation of $\{1,\ldots,n\}$. In the rest of the proof we consider only those labels. Now we can understand the $i$-th column of $A_2$ as the characteristic vector of neighbourhood of a vertex labeled by $i$. First, already proven part of this theorem tells us, that since $\Gamma_1$ and $\Gamma_2$ are TF-isomorphic, we have $N(\Gamma_1) = N(\Gamma_2)$. Only possible change is that sets from $N(\Gamma_1)$ can become neighbourhoods of different vertices than before. Therefore $m_{p} A_1  m_{p}^{-1} = A_2 m_{q}^{-1}$, where $q$ is the permutation of graph neighborhoods induced by $\pi_0$ with respect to labels $\psi_1 \circ \pi_0$ and $\psi_2$. Converse follows immediately form the first part of the theorem.
\end{proof}

    It is worth mentioning, that the next theorem was proven in \cite{zbMATH06719978} as Proposition 8. Stating this theorem in language of matrices and image of function $\alpha$ will give us interesting insight which will lead to equations connecting number of special type of antymorphisms with number of automorphisms of TF-isomorphic graphs to a given reduced one.

\begin{thm} \label{TW: kiedy TF-izo sa izo za pomoca Im(alpha)}
    Assume that $\Gamma_1 = (\text{V}_1,\text{E}_1)$ and $\Gamma_2 = (\text{V}_2,\text{E}_2)$ are reduced graphs with $n$ vertices and adjacency matrices $A_1$ and $A_2$ respectively. If there exist $p,q \in \text{Sym}(\{1,\ldots,n\})$ such that $m_{p}^{-1} A_2  m_{p} = A_1 m_{q}$, then $q \in \text{Ant}(\Gamma_1)$ and the following equivalence holds
    $$
    \Gamma_1 \cong \Gamma_2 \quad \Longleftrightarrow \quad q \in \text{Im}({\alpha}_{\Gamma_1}).
    $$
\end{thm}
\begin{proof}
    For proving that $q \in \text{Ant}(\Gamma_1)$ it is enough to note that $m_{p}^{-1} A_2  m_{p}$ is symmetric, therefore
    $$
    A_1 m_{q} = m_{p}^{-1} A_2  m_{p} = {(m_{p}^{-1} A_2  m_{p})}^{\text{T}} = {(A_1 m_{q})}^{\text{T}} = m_{q}^{-1} A_1. 
    $$
    Applying Proposition \ref{PROP: Aut^TF zdefiniowane macierzami} to the graph $\Gamma_1$ and permutation $q$ proves the result.
    \newline
    To prove the equivalence $\Gamma_1 \cong \Gamma_2 \Longleftrightarrow q \in \text{Im}({\alpha}_{\Gamma_1})$, let us start by noting that $\Gamma_1 \cong \Gamma_2$ if and only if there exists $r \in \text{Sym}(\{1,\ldots,n\})$ such that ${m_r}^{-1} A_1 m_r = A_2$. 
    \newline
    ,,$\Longrightarrow$''
    \newline
    Assumed relation of adjacency matrices allows us to obtain a relation between permutations $p,q$ and $r$, since $A_1 m_{q} = m_{p}^{-1} A_2  m_{p} = m_{p}^{-1} {m_r}^{-1} A_1 m_r  m_{p} = {m_{r \circ p}}^{-1} A_1 m_{r \circ p}$, which gives us ${m_{r \circ p}}^{-1} A_1 = A_1 m_{q} {m_{r \circ p}}^{-1}$, consequently $\gamma_{\Gamma_1}({(r \circ p)}^{-1}) = q \circ {(r \circ p)}^{-1}$. It leads us to equation $q = {\gamma_{\Gamma_1}(r \circ p)}^{-1} \circ (r \circ p) = \alpha_{\Gamma_1}(\gamma_{\Gamma_1}(r \circ p)) \in \text{Im}(\alpha_{\Gamma_1})$, which shows that indeed $q \in \text{Im}(\alpha_{\Gamma_1})$.
    \newline
    ,,$\Longleftarrow$''
    \newline
    Since $q \in \text{Im}({\alpha}_{\Gamma_1})$, there exists $r \in \text{Aut}^{\pi}(\Gamma_1)$ such that $q = \alpha_{\Gamma_1}(\gamma_{\Gamma_1}(r)) = \gamma_{\Gamma_1}(r) \circ r^{-1}$. This gives us $\gamma_{\Gamma_1}(r) = q \circ r$, and
    $$
    {m_{p \circ r}}^{-1} A_2 m_{p \circ r} = {m_{r}}^{-1} {m_{p}}^{-1} A_2 m_{p} m_{r} = {m_{r}}^{-1} A_1 m_{q} m_{r} = {m_{r}}^{-1} A_1 m_{q\circ r} = A_1.
    $$
    This completes the proof, as conjugation by a permutation matrix corresponds to relabeling of vertices.
\end{proof}

Next proposition provides an useful criteria for deciding, if there exists a graph with adjacency matrix $A_{\Gamma} \cdot m_{\psi}$, where $\Gamma$ is some reduced graph and $\psi$ is a permutation of its' verticies. 

\begin{prop} \label{PROP: warunek na istnienie grafu TF-izo}
    Let $\Gamma = (V,E)$ be a reduced graph with adjacency matrix $A$ and $\psi$ be a permutation of set $V$. Then the following conditions are equivalent:
    \begin{itemize}
        \item there exists a graph $\Gamma_{\psi}$ with adjacency matrix $A \cdot m_{\psi}$,
        \item $(v,\psi(v)) \notin E$ for every vertex $v \in V$ and $\psi \in \text{Ant}(\Gamma)$.
    \end{itemize}
    Moreover, if $\Gamma_{\psi}$ exists, it is a reduced graph which is TF-isomorphic to $\Gamma$.
\end{prop}
\begin{proof}
    Let us start with proving, that the first claim implies the second one. So, $A \cdot m_{\psi}$ have the same rows as $A$, but permuted by permutation $\psi^{-1}$. Since $A \cdot m_{\psi}$ have only zeros on diagonal, we know that $\psi^{-1}$ maps every vertex to a vertex which it is not connected too, otherwise there would be ones on diagonal. Also, since $A \cdot m_{\psi}$ is a symmetric matrix, we get $A \cdot m_{\psi} = {(A \cdot m_{\psi})}^{T} = m_{{\psi}^{-1}} \cdot A$.
    \newline
    Now let us prove the converse. For a matrix to be an adjacency matrix of some graph, its entries have to be zeros or ones, it have to be symmetric and its diagonal must consist only of zeros.
    \newline
    Entries of $A \cdot m_{\psi}$ definitely consist only of zeros and ones, since $m_{\psi}$ is a permutation matrix. It is symmetric, since $\psi \in \text{Ant}(\Gamma)$, and therefore
    $$
    {(A \cdot m_{\psi})}^{\text{T}} = m_{\psi^{-1}} \cdot A = m_{\gamma(\psi)} \cdot A = A \cdot m_{\psi}.
    $$
    To prove that there are no zeros on diagonal it is enough once again use an argument that $A \cdot m_{\psi}$ have the same rows as $A$, but they are permuted by $\psi^{-1}$.
\end{proof}

    For any finite group $\text{H}$ and any $h \in \text{H}$ we will use symbol $\varphi_h$, to denote an automorphism of $\text{H}$ given by formula $\varphi_h : \text{H} \ni q \mapsto h^{-1} q h \in \text{H}$.

\begin{prop} \label{PROP: co sie dzieje z gammą przy przeksztalacaniu grafu na TF-izo}
    Let $\Gamma_1$ be a reduced graph with adjacency matrix $A_1$, $\Gamma_2$ be a graph with adjacency matrix $A_2$ and let $\psi \in \text{Ant}(\Gamma_1)$. If $A_2 = A_1 \cdot m_{\psi}$, then $\text{Aut}^{\pi}(\Gamma_1) = \text{Aut}^{\pi}(\Gamma_2)$ and $\gamma_{\Gamma_2} = \varphi_{\psi} \circ \gamma_{\Gamma_1}$.
\end{prop}
\begin{proof}
    Take any $\pi \in \text{Aut}^{\pi}(\Gamma_1)$. Then, we get $m_{\pi} A_1 = A_1 m_{\gamma_{\Gamma_1}(\pi)}$. This gives
    $$
    m_{\pi} \cdot A_2 = m_{\pi} A_1 \cdot m_{\psi} = m_{\pi} \cdot A_1 = A_1 \cdot m_{\gamma_{\Gamma_1}(\pi)} m_{\psi} = A_2 \cdot m_{\psi}^{-1} m_{\gamma_{\Gamma_1}(\pi)} m_{\psi} = A_2 \cdot m_{\varphi_{\psi} \circ \gamma_{\Gamma_1}(\pi)}.
    $$
    Applying Proposition \ref{PROP: Aut^TF zdefiniowane macierzami} yields the result about $\gamma_{\Gamma_2}$ for element $\pi$.
    It also shows that $\text{Aut}^{\pi}(\Gamma_1) \subseteq \text{Aut}^{\pi}(\Gamma_2)$. To get the reverse it is enough to see that $A_2 \cdot m_{\psi^{-1}} = A_1$.
\end{proof}

\section{Notes on nonisomorphic graphs which are TF-isomorphic to a given graph} \label{SECTION: Notes on nonisomorphic graphs which are TF-isomorphic to a given graph}

Pacco and Scapellato \cite[Theorem 3,5]{pr1997} characterized nonisomorphic graphs TF-isomorphic to a given graph using certain conjugacy classes from $\text{Aut}(\text{B}\Gamma)$. To understand this theorem we first need to state some definitions and lemmas. We will strengthen their result by additionally proving, that sizes of those conjugacy classes are equal to instability indexes of corresponding graphs. In this Section we will use the following convention. For a group $\text{H}$, and $h \in \text{H}$, by ${(h)}_{c}$ we understand the conjugacy class of $h$ in $\text{H}$.

\begin{defi}
    Let $\Gamma$ be a connected reduced nonbipartite graph. W denote by $\text{S}(\Gamma)$ the set of all $g \in \text{Aut}(\text{B}\Gamma)$ of order two and which switch bipartite parts of $\text{B}\Gamma$.
\end{defi}

\begin{defi}
    Let $\Gamma$ be a connected reduced nonbipartite graph and let $g \in \text{S}(\Gamma)$. Then $g$ is called $\textit{strongly switching element of }\Gamma$ if and only if every vertex of $\text{B}\Gamma$ is mapped by $g$ to a vertex outside of its neighbourhood.
\end{defi}

For any connected reduced and nonbipartite graph $\Gamma$ we will use an isomorphism $\text{Aut}(\text{B}\Gamma) \cong \text{Aut}^{\pi}(\Gamma) \rtimes_{\theta} {\mathbb{Z}}_2$, where ${\mathbb{Z}}_2 = \{e,x\}$ and $\theta(x) \equiv \gamma$. For convenience, we refer to the element $(e,x) \in \text{Aut}^{\pi}(\Gamma) \rtimes_{\theta} {\mathbb{Z}}_2$ simply by $x$.
\newline
 In the next lemma we will use a function $x \cdot$ to map a set $\text{Ant}(\Gamma)$ to the set $\text{S}(\Gamma)$. One should note that formality, function $x \cdot$ first take elements from $\text{Ant}(\Gamma)$, which is apriori subset of $\text{Aut}^{\pi}(\Gamma)$, to $\text{Aut}^{\pi}(\Gamma) \rtimes_{\theta} {\mathbb{Z}}_2$ in a canonical way, and then multiply by $x$ on the left.
 
\begin{lem} \label{LEMAT: bijekcja Antymorfizmy <-> Swithing elementy}
    Let $\Gamma$ be a connected reduced nonbipartite graph. Then the function $x \cdot $ defines a bijection from $\text{Ant}(\Gamma)$ to $\text{S}(\Gamma)$.
\end{lem}
\begin{proof}
    Clearly function $x \cdot$ is injective. To check that it maps $\text{Ant}(\Gamma)$ to $\text{S}(\Gamma)$ it is enough to perform following calculation
    $$
    e = {(x \cdot g)}^2 = (x \cdot g \cdot x) \cdot g = \gamma(g) \cdot g \quad \Longleftrightarrow \quad \gamma(g) = g^{-1}.
    $$  
\end{proof}


\begin{lem} \label{g i g' z SGamma sprzezone --> g^{-1}g' w Im(alpha) Gamma_g}
    Let $\Gamma$ be a connected reduced nonbipartite graph and let $g,g' \in S(\Gamma)$. Then  $g' \in {(g)}_{c}$ if and only if $g^{-1} g' \in \text{Im}(\alpha_{g})$, where $\alpha_g$ is the function $\text{Aut}^{\pi}(\Gamma) \rightarrow \text{Aut}^{\pi}(\Gamma)$ defined by a formula $h \mapsto h^{-1} \cdot \varphi_{x \cdot g} \circ \gamma(h)$.
\end{lem}
\begin{proof}
    Recal that $\text{Aut}(\text{B}\Gamma) \cong \text{Aut}^{\pi}(\Gamma) \rtimes_{\theta} {\mathbb{Z}}_2$. We will first prove that $g' \in {(g)}_{c} \Rightarrow g^{-1} g' \in \text{Im}(\alpha_{g})$. By Lemma \ref{LEMAT: bijekcja Antymorfizmy <-> Swithing elementy} write $g = x \cdot g_0$ and $g' = x \cdot {g'}_0$, where $g_0,g'_0 \in \text{Ant}(\Gamma)$. Since $g' \in {(g)}_{c}$, there exists an element $h$ from $\text{Aut}(\text{B}\Gamma)$ such that $g' = h^{-1} g h$, and $x^{-1} g x = x g_0^{-1} = {g_0} x {g_0} g_0^{-1} =  {g_0} g g_0^{-1}$ hence, we can assume that $h \in \text{Aut}^{\pi}(\Gamma)$. Therefore $x \cdot g'_0 = g' = h^{-1} g h = h^{-1} x g_0 h = x {\gamma(h)}^{-1} g_0 h$, which implies $g'_0 = {\gamma(h)}^{-1} g_0 h$. The following calculation finishes the proof of first implication
    $$
    g^{-1} g' = g_0^{-1} g'_0 = g_0^{-1} {\gamma(h)}^{-1} g_0 h = \alpha_g(g_0 h g_0^{-1}).
    $$
    To prove that $g^{-1} g' \in \text{Im}(\alpha_{g}) \Rightarrow g' \in {(g)}_{c}$ it is enough to observe, that there exists $h \in \text{H}$ such that $ g_0^{-1} g'_0 = g^{-1} g' = \alpha_g(g_0 h g_0^{-1}) = {g_0}^{-1} {\gamma(h)}^{-1} g_0 h$, so $g'_0 = {\gamma(h)}^{-1} g_0 h$ and equivalently $g' = h^{-1} g h$.
\end{proof}

Following proposition is a known result, however for completeness of this paper we will see a short a short proof.

\begin{prop} \label{PROP: strongly switching elementy to klasy sprzezen}
    Let $\Gamma = (V,E)$ be a connected reduced nonbipartite graph, and $g \in \text{Aut}(\text{B}\Gamma)$ be a strongly switching element. Then every $g' \in {(g)}_{c}$ is also a strongly switching element.
\end{prop}
\begin{proof}
    Since $\Gamma$ is reduced and nonbipartite we get $\text{Aut}(\text{B}\Gamma) \cong \text{Aut}^{\pi}(\Gamma) \rtimes_{\theta} {\mathbb{Z}}_2$ where ${\mathbb{Z}}_2 = \{e,x\}$ and $\theta(x) = \gamma$. By lemma \ref{LEMAT: bijekcja Antymorfizmy <-> Swithing elementy} there exists $g_0,g'_0 \in \text{Ant}(\Gamma)$ such that $g_0 = x \cdot g$ and $g'_0 = x \cdot g'$. The fact that $g$ is strongly switching translates to the following. For every $v \in V$ it holds that $(v,g_0(v)) \notin E$. Therefore, Proposition \ref{PROP: warunek na istnienie grafu TF-izo} assures us, that there is a graph $\Gamma_0$ with adjacency matrix $A_{\Gamma} \cdot m_{g_0}$. Now, by Lemma \ref{g i g' z SGamma sprzezone --> g^{-1}g' w Im(alpha) Gamma_g} and Proposition \ref{PROP: co sie dzieje z gammą przy przeksztalacaniu grafu na TF-izo} $g_0^{-1} g'_0 \in \text{Im}(\alpha_{\Gamma_0})$. By Corollary \ref{COR: brak krawedzi w orbitach alphy}, Observation \ref{OBS wlasnosci alphy i elementow odwracanych przez gamme} (a) and Proposition \ref{PROP: warunek na istnienie grafu TF-izo} applied to $\Gamma_0$ we conclude that there exists a graph $\Gamma'_0$ with adjacency matrix $A_{\Gamma'_0} = A_{\Gamma_0} \cdot m_{{g_0}^{-1} g'_0} = A_{\Gamma} \cdot m_{g_0} \cdot m_{ g_0^{-1} g'_0} = A_{\Gamma} \cdot m_{g'_0}$. Applying Proposition \ref{PROP: warunek na istnienie grafu TF-izo} to $\Gamma$ shows, that $g'_0$ maps every vertex of $\Gamma$ to a vertex it was not connected to. Therefore $g'$ is a strongly switching element.
\end{proof}

Now we are ready to prove the main result of this section.

\begin{thm} \label{THM: klasy sprzezenosci strongly switching mapuja sie na TF-iso grafy a ilosc ich elementow to instability indexy}
    Let $\Gamma$ be a connected reduced nonbipartite graph. Then there is a bijection $\psi_{\Gamma}$ between family of conjugacy classes consisting of strongly switching elements of $\Gamma$ and the set of nonisomoprhic graphs TF-isomorphic to $\Gamma$. Moreover, cardinality of any conjugacy class $\mathcal{C}$ which consist of strongly switching elements is equal to the instability index of a graph $\psi_{\Gamma}(\mathcal{C})$.
\end{thm}
\begin{proof}
    We will construct a function from strongly switching elements od $\Gamma$ to set of nonisomoprhic graphs TF-isomorphic to $\Gamma$. Then we will show, that this function agrees on two elements exactly if they are conjugate. We will denote the set of strongly switching elements of $\Gamma$ by $S_0(\Gamma)$.
    \newline
    By Lemma \ref{LEMAT: bijekcja Antymorfizmy <-> Swithing elementy} function $x \cdot$ maps $S(\Gamma)$ to $\text{Ant}(\Gamma)$. Observe, that strongly switching elements are exactly the preimages of elements from $\text{Ant}(\Gamma)$ which maps every vertex $v$ to a vertex from $V \backslash N(v)$. Then, by Proposition \ref{PROP: warunek na istnienie grafu TF-izo}, for every $g \in S_0(\Gamma)$, there exists a graph with adjacency matrix $A_{\Gamma} \cdot m_{x \cdot g}$. We denote the set of unlabeled graphs TF-isomorphic to $\Gamma$ by $\mathcal{G}_{\text{TF}}(\Gamma)$ and the map described above by $G_{\Gamma}: S_0(\Gamma) \longrightarrow \mathcal{G}_{\text{TF}}(\Gamma)$.
    \newline
    By Proposition \ref{PROP: przetlumaczenie TF-izo na macierze i sasiedztwa}, up to permutation of vertices, every graph TF-isomorphic to $\Gamma$ have an adjacency matrix of the form $A \cdot m_{\pi}$ for some permutation of vertices $\pi$. Moreover by Proposition \ref{PROP: warunek na istnienie grafu TF-izo} $\pi$ have to be of the form $x \cdot g$ for some $g \in S_0(\Gamma)$. Hence $G_{\Gamma}$ is a surjection.
    \newline
    Next step in our proof is showing, that function $G_{\Gamma}$ is equal on two elements exactly where those are conjugates. Firstly, by Theorem \ref{TW: kiedy TF-izo sa izo za pomoca Im(alpha)} we conclude that graphs $\Gamma_1$ and $\Gamma_2$ with adjacency matricies $A_{\Gamma} \cdot m_{x \cdot g_1}$ and $A_{\Gamma} \cdot m_{x \cdot g_2}$ are isomorphic if and only if $ {(x \cdot g_1)}^{-1} (x \cdot g_2) \in \text{Im}(\alpha_{\Gamma_1})$. By Proposition \ref{PROP: co sie dzieje z gammą przy przeksztalacaniu grafu na TF-izo} we get that $\gamma_{\Gamma_1} = \varphi_{x \cdot g_1} \circ \gamma_{\Gamma}$, and therefore $\alpha_{\Gamma_1} \equiv \alpha_{x \cdot g_1}$. Last equality indicates that $ {(x \cdot g_1)}^{-1} (x \cdot g_2) \in \text{Im}(\alpha_{x \cdot g_1})$. By Lemma \ref{g i g' z SGamma sprzezone --> g^{-1}g' w Im(alpha) Gamma_g} this is equivalent to the fact that $g_2 \in {(g_1)}_c$.
    \newline 
    Now we know, that $G_{\Gamma}$ is equal on conjugacy classes. We denote the family of conjugacy classes of strongly switching elements by $C_0
    (\Gamma)$. Now we can define a function $\psi_{\Gamma}: C_0(\Gamma) \rightarrow \mathcal{G}_{\text{TF}}(\Gamma)$ by formula $\psi_{\Gamma}: {(g)}_c \mapsto A_{\Gamma} \cdot m_{x\cdot g}$.
    \newline
    Let $\mathcal{C} = (g)_c \in C_0(\Gamma)$. It is bijected by $x \cdot$ to the set $G_0$ of elements $g'_0$ in $\text{Ant}(\Gamma)$ such that $g_0^{-1} g'_0 \in \text{Im}(\alpha_{g_0}) = \text{Im}(\alpha_{G_{\Gamma}(g_0)})$, where $g_0 = x \cdot g$. Therefore we get that $|(g)_c| = |\text{Im}(\alpha_{\psi_{\Gamma}(g_0)})|$, and by Theorem \ref{THM: alpha jest stala na prawostronnych warstwach Aut} we get that it is equal to the instability index of the graph $G_{\Gamma}(g) = \psi_{\Gamma}(\mathcal{C})$, which finishes the proof.
\end{proof}

\begin{cor} \label{COR: rownania dotyczace TF-iso grafow}
    Let $\Gamma$ be a connected reduced nonbipartite graph and let $\text{Ant}_0(\Gamma)$ denote the set of elements from $\text{Ant}(\Gamma)$ which maps vertices to vertices, which are not their neighbours. Then
    $$
    |\text{Ant}_0(\Gamma)| = \sum_{\Gamma' \cong^{\text{TF}} \Gamma} \text{Inst}(\Gamma') \quad \text{and} \quad \frac{|\text{Ant}_0(\Gamma)|}{|\text{Aut}^{\pi}(\Gamma)|} = \sum_{\Gamma' \cong^{\text{TF}} \Gamma} \frac{1}{|\text{Aut}(\Gamma')|}.
    $$
\end{cor}
\begin{proof}
    For the first equality is enough, note that the set of strongly switching elements have cardinality $|\text{Ant}_0(\Gamma)|$ and $\text{Inst}(\Gamma')$ is the cardinality of the preimage of $\Gamma'$ by the function $\psi_{\Gamma}$. To get the second equality it is enough to divide both sides by $|\text{Aut}^{\pi}(\Gamma)|$.
\end{proof}

\begin{obs} \label{OBS: rownania dotyczace TF-izo ale z grafami z petlami}
    If one include the possibility of graphs with loops, then by similar method as in Theorem \ref{THM: klasy sprzezenosci strongly switching mapuja sie na TF-iso grafy a ilosc ich elementow to instability indexy} for any reduced connected nonbipartite graph $\Gamma$, equalities below follow
    $$
    |\text{Ant}(\Gamma)| = \sum_{\Gamma' \cong^{\text{TF}} \Gamma} \text{Inst}(\Gamma') \quad \text{and} \quad \frac{|\text{Ant}(\Gamma)|}{|\text{Aut}^{\pi}(\Gamma)|} = \sum_{\Gamma' \cong^{\text{TF}} \Gamma} \frac{1}{|\text{Aut}(\Gamma')|}.
    $$
\end{obs}

Next result is a straightforward application of Corollary \ref{COR: rownania dotyczace TF-iso grafow}, however it will lead us to an interesting question.

\begin{cor}
    Let $\mathcal{G}$ be a finite family of unlabeled reduced connected nonbipartite graphs such that if $\Gamma \cong^{\text{TF}} \Gamma'$ and $\Gamma \in \mathcal{G}$, then $\Gamma' \in \mathcal{G}$. Then the following holds
    $$
    \mathbb{E}(|\text{Aut}(\Gamma)|) \geq \mathbb{E}(|\{\Gamma' \mid \Gamma' \cong^{\text{TF}} \Gamma \}|).
    $$
    By $\mathbb{E}$'' we mean the expected value, or equivalently the arithmetic mean over all elements $\Gamma$ from $\mathcal{G}$. 
\end{cor}
\begin{proof}
    We can divide $\mathcal{G}$ with respect to the relation ''$\cong^{\text{TF}}$''. We prove the statement of the theorem for each equivalence class on its own. Let $\mathcal{G}_0$ be a such equivalence class. We can write the following inequality
    $$
    1 \geq \frac{|\text{Ant}(\Gamma)|}{|\text{Aut}^{\pi}(\Gamma)|} = \sum_{\Gamma' \cong^{\text{TF}} \Gamma} \frac{1}{|\text{Aut}(\Gamma')|} \geq |\{\Gamma' \mid \Gamma' \cong^{\text{TF}} \Gamma \}| \cdot \frac{1}{\mathbb{E}(|\text{Aut}(\Gamma)|)} \text{ for } \mathcal{G}_0.
    $$
    First inequality follows from the fact that $\text{Ant}(\Gamma) \subseteq \text{Aut}^{\pi}(\Gamma)$, equality in the middle is the statement of Corollary \ref{COR: rownania dotyczace TF-iso grafow} and the inequality on the right follows from the inequality between harmonic and arithmetic mean. Multiplying both sides by $\mathbb{E}(|\text{Aut}(\Gamma)|)$ proves that 
    $$
    \mathbb{E}(|\text{Aut}(\Gamma)|) \geq |\{\Gamma' \mid \Gamma' \cong^{\text{TF}} \Gamma \}| = \mathbb{E}(|\{\Gamma' \mid \Gamma' \cong^{\text{TF}} \Gamma \}|),
    $$
    since $\mathcal{G}_0$ consists only of graphs which are TF-isomorphic to $\Gamma$.
\end{proof}

    The above result suspects the following
\begin{que} \label{QUE: czy zawsze Aut > TF-iso graphs?}
    Is it true that for all connected reduced nonbipartite graphs $\Gamma$ that 
    $$
    |\text{Aut}(\Gamma)| \geq |\{\Gamma' \mid \Gamma' \cong^{\text{TF}} \Gamma \}|?
    $$ 
\end{que}

\section{Construction of (H,$\sigma$)-graphs with any number and type of TF-isomorphic graphs} \label{SECTION: Construction of (H,sigma)-graphs with any number and type of TF-isomorphic graphs}

In this section we will show, that we are able to manipulate the number and the type of graphs TF-isomorphic to a given one. To state our main theorem we have to describe $S(\Gamma)$ in terms of the group $\text{H} \rtimes_{\theta} \mathbb{Z}_2$.

\begin{defi}
     Let $\text{H}$ be a finite group and $\sigma \in \text{Aut}(\text{H})$ be an automorphism of $\text{H}$ such that $\sigma^{2} \equiv Id$. By $S(\text{H},\sigma)$ we understand the subset of elements $g$ from a group $\text{H} \rtimes_{\theta} \mathbb{Z}_2$ such that $g \notin \text{H} \times \{e\}$ and $g^2 = e$. We will stick to the convention that $\mathbb{Z}_2 = \{e,x\}$ and the semidirect product $\text{H} \rtimes_{\theta} \mathbb{Z}_2$ is defined by $\theta(x) \equiv \sigma$.
\end{defi}

We also refer to subgroup $\text{H} \times \{e\}$ of $\text{H} \rtimes_{\theta} \mathbb{Z}_2$ by $\text{H}$ for simplicity.

\begin{defi}
    Let $\Gamma$ be a connected reduced nonbipartite graph. We say, that $\Gamma$ \textit{achieves} subset $\mathcal{C}$ of $S(\text{H},\sigma)$ if and only if $\Gamma$ is a $(\text{H},\sigma)$-graph and $\mathcal{C}$ is exactly the image of strongly switching elements of $\Gamma$ after the isomorphism from $\text{Aut}(\text{B}\Gamma)$ to $\text{H} \rtimes \mathbb{Z}_2$.
    \newline
    Let ${\mathcal{C}}$ be a subset of $S(\text{H},\sigma)$. We call $\mathcal{C}$ \textit{achievable} if and only if there exist a connected reduced nonbipartite graph $\Gamma$ which achieves $\mathcal{C}$.
\end{defi}

Observation below narrows our set of possible achievable subsets of $S(\text{H},\sigma)$.

\begin{obs} \label{OBS: every achivable subset of S(H,sigma) is sum of conj classes}
    Every achievable subset $\mathcal{C}$ of $S(\text{H},\sigma)$ is a sum of conjugacy classes from $\text{H} \rtimes_{\theta} \mathbb{Z}_2$ which contains $x$.
\end{obs}
\begin{proof}
    It is enough, to apply Proposition \ref{PROP: strongly switching elementy to klasy sprzezen} for any graph which realizes $\mathcal{C}$ and note, that $x\cdot x = e$ and $A_{\Gamma} \cdot m_{e} = A_{\Gamma}$ for any graph.
\end{proof}

For the purpose of our next theorem we have to introduce a concept of the Generalized Cayley Graphs which were previously studied in \cite{zbMATH00064637} and \cite{zbMATH07081656}.

\begin{defi} \label{DEFI: Generalized Cayley Graphs}
    Let $\text{H}$ be a finite group and let $\sigma \in \text{Aut}(\text{H})$ be such that $\sigma^{2} \equiv Id$. Let also $S \subseteq \text{H} \backslash \text{Im}(\alpha_{\gamma})$ be such that $S = {\sigma(S)}^{-1}$. Then $\text{GCay}(\text{H},\sigma,S)$ is a graph with vertices labeled by elements of $\text{H}$ and edges between $h_1$ and $h_2$ if and only if ${h_1}^{-1} \cdot {\sigma(h_2)} \in S$.   
\end{defi}

It is well known, that $\text{H} \leq \text{Aut}^{\pi}(\text{GCay}(\text{H},\sigma,S))$ and $\gamma {|}_{\text{H}} \equiv \sigma$. We will use Generalized Cayley Graphs to modify the construction from Theorem \ref{konstrukcja (H,sigma) grafow} in such a way, that any sum of conjugacy classes from $S(\text{H},\sigma)$ will be achieved by our newly constructed graph. 

\begin{thm}
    Let $\text{H}$ be a finite group and let $\sigma \in \text{Aut}(\text{H})$ be such that $\sigma^{2} \equiv Id$. Let $\mathcal{C}$ be a sum of conjugacy classes from $S(\text{H},\sigma)$ which contains $x$ (nontrivial element of $\mathbb{Z}_2$). Then $\mathcal{C}$ is achievable.
\end{thm}
\begin{proof} 
    Let $S = x \cdot ( S(\text{H},\sigma) \backslash \mathcal{C} )$, and note that $\sigma(S) = S^{-1}$, because by Lemma \ref{LEMAT: bijekcja Antymorfizmy <-> Swithing elementy} it is true pointwise. Therefore ${\sigma(S)}^{-1} = {(S^{-1})}^{-1} = S$. Consider now a graph $\Gamma_{(\text{H},\sigma), X}$ ( cf. Definition \ref{DEFI: Gamma_{H,sigma, X}}).
    Add edges of $\text{GCay}(\text{H},\sigma,S)$ between vertices $\text{H} \times \{n_0 + 2\}$ with respect to projection on the first factor. We will refer to this graph as $\Gamma_0$.
    \newline
    The vertices in $\text{H} \times \{n_0 + 2\}$ are the only ones with the highest degree among all vertices of $\Gamma_0$. Therefore, every permutation form $\text{Aut}^{\pi}(\Gamma_0)$ permutes the set $\text{H} \times \{n_0 + 2\}$.  This means, that  if $\pi_0 \in \text{Aut}^{\pi}(\Gamma_0)$ then $\pi_0 \in \text{Aut}^{\pi}(\Gamma_{(\text{H},\sigma), X})$ since neighbourhoods of all vertices which are not in $\text{H} \times \{n_0 + 2\}$ are the same in those graphs, and any neighbourhood $N(v_0)$ of vertex from $\text{H} \times \{n_0 + 2\}$ can be break into two parts. First one consists of vertices which are in $\text{H} \times \{n_0 + 2\}$ and the second one consists of all other vertices inside $N(v_0)$. All vertices in $N(v_0)$ which are in $\text{H} \times \{n_0 + 2\}$ are exactly ones obtained from edges of $\text{GCay}(\text{H},\sigma,S)$. Since we removed all of them and $\pi_0(N(v_0)) = N(\gamma(\pi_0)(v_0))$ is a neighbourhood of some vertex from $\text{H} \times \{n_0 + 2\}$, $\pi_0$ maps all neighbourhoods of $\Gamma_{(\text{H},\sigma), X}$ to neighbourhoods of $\Gamma_{(\text{H},\sigma), X}$.
    \newline
    From Theorem \ref{konstrukcja (H,sigma) grafow} we get that, $\text{Aut}^{\pi}(\Gamma_0) \leq \text{Aut}^{\pi}(\Gamma_{(\text{H},\sigma), X}) \cong \text{H}$, but since added verticies form a genralized Cayley graph, $\text{Aut}^{\pi}(\Gamma_0) = \text{Aut}^{\pi}(\Gamma_{(\text{H},\sigma), X})$. To end the proof we have to show, that $\Gamma_0$ indeed achieves $\mathcal{C}$.  For that purpose we will apply Proposition \ref{PROP: warunek na istnienie grafu TF-izo}. No element $s_0 \in x \cdot (S(\text{H},\sigma) \backslash \mathcal{C})$ corresponds to a graph, since there is and edge between vertices $(e,n_0+2)$ and $(s_0,n_0+2)$ by definition of $\text{GCay}(\text{H},\sigma,S)$. Pick any $h_0 \in x \cdot \mathcal{C}$. Take vertex $(h,n_0+2)$ for any $h \in \text{H}$. Then vertices $(h_0 \cdot h,n_0+2)$ and $(h,n_0+2)$ are connected if and only if ${(h_0 h)}^{-1}\sigma(h) = h^{-1} \cdot h_0 \cdot \sigma(h) \in S$. However, since $h_0 \notin S$, we get that $x\cdot h_0 \notin S(\text{H},\sigma) \backslash \mathcal{C}$, and since a set $S(\text{H},\sigma) \backslash \mathcal{C}$ is sum of conjugacy classes, ${\sigma(h)}^{-1} (x \cdot h_0) \sigma(h) = x \cdot (h^{-1} \cdot h_0 \cdot \sigma(h)) \notin S(\text{H},\sigma) \backslash \mathcal{C}$. Therefore $h^{-1} \cdot h_0 \cdot \sigma(h) \notin S$. This shows that there exists a graph with adjacency matrix $A_{\Gamma_0} \cdot m_{h_0}$.
\end{proof}

\section{Further notes on nonisomorphic graphs which are TF-isomorphic to a given graph} \label{SECTION: Further notes on nonisomorphic graphs which are TF-isomorphic to a given graph}

    Now we investigate how much TF-isomorphic graphs can $(\text{H},\sigma)$-graph possibly have. We denote the number of distinct conjugacy classes inside $S(\text{H},\sigma)$ by $\mathcal{G}_{\text{TF}}(\text{H},\sigma)$, as every conjugacy class might correspond to one graph.

\begin{thm} \label{THM: ogarniczenie na ilosc TF-izo grafow - 2-grupa Sylowa}
     Let $\text{H}$ be a finite group and $\sigma \in \text{Aut}(\text{H})$ be such that $\sigma^2 \equiv Id$. Then there exists such $2$-Sylow subgroup $P_2 \leq \text{H}$, such that $\mathcal{G}_{\text{TF}}(\text{H},\sigma) \leq \mathcal{G}_{\text{TF}}(P_2,\sigma {|}_{P_2})$.
\end{thm}
\begin{proof}
    Consider the set $\mathcal{P}_2$ of all Sylow $2$-subgroups of $\text{H}$. Function $\sigma_{\mathcal{P}_2} : \mathcal{P}_2 \longrightarrow \mathcal{P}_2$ acts by applying $\sigma$ setwise. Since $\sigma^2 \equiv Id$ and $\mathcal{P}_2$ have an odd number of elements, $\sigma_{\mathcal{P}_2}$ have a fixed point. Call this fixed point $P_2$. Consider a group $\text{H} \rtimes \mathbb{Z}_2$, $\mathbb{Z}_2 = \{e,x\}$, and a set $Q_2 = P_2 \times \{0,1\}$. For any $h \in \text{H}$ we have $x^{-1} \cdot h \cdot x = \sigma(h)$ and $P_2$ is fixed setwise by $\sigma$ hence $Q_2$ is a group. Since $P_2$ was a Sylow $2$-subgroup of $H$, $Q_2$ is a Sylow $2$-subgroup of $\text{H} \rtimes \mathbb{Z}_2$. 
    \newline
    Let $s_0$ be an element from $S(\text{H},\sigma)$. Since it is of order two, it is contained in some Sylow $2$-subgroup $Q'_2$. From the fact, that all Sylow $2$-subgroups are conjugates of each other we deduce, that there exists $s'_0 \in Q_2$ such that $s'_0 \in {(s_0)}_c$, and therefore $s'_0 \in S(\text{H},\sigma)$. Therefore every conjugacy class inside $S(\text{H},\sigma)$ contains an element from $Q_2$. The number of conjugacy classes with respect to $\text{H}$ inside $S(\text{H},\sigma)$ is equal to number of conjugacy classes inside $S(\text{H},\sigma) \cap Q_2$, and therefore is less or equal to the number of conjugacy classes inside $S(P_2,\sigma {|}_{P_2})$ with respect to $P_2$, because $P_2 \leq \text{H}$.
\end{proof}

Above theorem shows us, that it is sufficient to consider Question \ref{QUE: czy zawsze Aut > TF-iso graphs?} in case of $2$-groups. Corollary below is useful whenever we do not know exactly how $\sigma$ acts on $\text{Aut}^{\pi}(\Gamma)$.

\begin{cor} \label{COR: ogaraniczenie na ilosc TF-izo - 2^k}
      Let $\text{H}$ be a finite group and $\sigma \in \text{Aut}(\text{H})$ be such that $\sigma^2 \equiv Id$. If $|\text{H}| = 2^k \cdot m$ where $k \in {\mathbb{Z}}_{\geq 0}$ and $m$ is odd, then $\mathcal{G}_{\text{TF}}(\text{H},\sigma) \leq 2^k$.
\end{cor}
\begin{proof}
    Theorem \ref{THM: ogarniczenie na ilosc TF-izo grafow - 2-grupa Sylowa} shows that $\mathcal{G}_{\text{TF}}(\text{H},\sigma) \leq \mathcal{G}_{\text{TF}}(P_2,\sigma {|}_{P_2})$, where $P_2$ is some $2$-Sylow subgroup of $\text{H}$. Since $x \cdot P_2$ have $2^k$ elements, and $S(P_2,\sigma {|}_{P_2}) \subseteq x \cdot P_2$, $\mathcal{G}_{\text{TF}}(P_2,\sigma {|}_{P_2})$ is bounded above by $2^k$.
\end{proof}

Even if we only have information about automorphism group of a given graph $\Gamma$, sometimes we can still prove, that there do not exist a graph TF-isomorphic and nonisomorphic to $\Gamma$. We should note that next proposition was laready known, however we present a different proof.

\begin{cor} \label{COR: brak aut rzedu 2 -> graf jest stabilny}
    \cite[Proposition 10]{zbMATH06719978} If $\Gamma$ is a connected reduced nonbipartite graph with an automorhism group of odd order, then there does not exist any nonisomorphic graph $\Gamma_0$ which is TF-isomorphic to $\Gamma$. We shall note, that it was previously known, but a proof was different.
\end{cor}
\begin{proof}
    By applying Corollary \ref{COR: parzystosc Aut taka sama jak parzystosc Aut^pi} the group $\text{Aut}^{\pi}(\Gamma)$ has odd order. Applying Corollary \ref{COR: ogaraniczenie na ilosc TF-izo - 2^k} ends the proof.
\end{proof}

Following results may help with designing algorithm, which would search for graphs TF-isomorphic to a given one. Some of them are known and we formulate them in a slightly different, but equivalent language.

\begin{lem} \label{LEM: sigma(h) = h^{-1} daje ze e i h^2 sa izo}
    \cite[Proposition 9]{zbMATH06719978} Let $\text{H}$ be a finite group and $\sigma \in \text{Aut}(\text{H})$ be such that $\sigma^2 \equiv Id$. If $h_1,h_2 \in \text{H}$ and $\sigma(h_1) = {h_1}^{-1}$, then $x \cdot h_1 h_2 h_1 \in {x\cdot h_2}_c$. In particular for any $h_0 \in \text{H}$ such that $\sigma(h_0) = {h_0}^{-1}$ and any $k \in \mathbb{Z}$ we obtain $x \cdot {h_0}^{2k} \in {(x)}_c$ and $x \cdot {h_0}^{2k+1} \in {(x\cdot h_0)}_c$. By $x$ we denote the nontrivial element of $\mathbb{Z}_2$ from $\text{H} \rtimes \mathbb{Z}_2$.
\end{lem}
\begin{proof}
    For the first part it is enough to notice that 
    ${h_1}^{-1} (x\cdot h_2) h_1 = x \cdot (h_1 h_2 h_1)$ since $\sigma(h_1) = {h_1}^{-1}$. Put $h_2 = e$ and $h_1 = {h_0}^k$ to get second claim of this lemma, and then set $h_2 = h_0$ and $h_1 = {h_0}^k$ to finish the proof.
\end{proof}

    For any finite group $\text{H}$ and its element $h \in \text{H}$ by $\text{ord}(h)$ we denote the smallest positive inteeger $n$ such that $h^n = e$.

\begin{thm} \label{THM: kazdy Main comoponent jest prepezentowany przez element rzedu 2^k}
     Let $\text{H}$ be a finite group and $\sigma \in \text{Aut}(\text{H})$ be such that $\sigma^2 \equiv Id$. Then every conjugacy class inside $S(\text{H},\sigma)$ contains an element $x \cdot h_0$ such that there exists $k \in \mathbb{Z}_{\geq 0}$ for which $\text{ord}(h_0) = 2^k$.
\end{thm}
\begin{proof}
    Let $x\cdot h_0$ be an element from $S(\text{H},\sigma)$. Let $\text{ord}(h_0) = 2^{k_0} \cdot (2{l_0}+1)$, where $k_0,l_0 \in \mathbb{Z}_{\geq 0}$. Applying Lemma \ref{LEM: sigma(h) = h^{-1} daje ze e i h^2 sa izo} we get $h_0^{2l_0 + 1} \in {(h_0)}_c$. To finish the proof it is enough to observe that $\text{ord}({h_0}^{2{l_0}+1}) = 2^{k_0}$.
\end{proof}

One can also deduce this result from Theorem \ref{THM: ogarniczenie na ilosc TF-izo grafow - 2-grupa Sylowa} rather than prove it by the method from \cite{zbMATH06719978}, however that would not show how to construct a representant of order $2^k$ from any given representant of a given conjugacy class from $S(\text{H},\sigma)$.

\begin{thm} \label{THM: jesli A m_pi jest grafem, to orbity pi to antykliki}
    Let $\Gamma = (V,E)$ be a connected reduced nonbipartite graph with adjacency matrix $A$. If there exists a graph with adjacency matrix $A \cdot m_{\pi}$, for some $\pi \in \text{Sym}(V)$, then subgraph of $\Gamma$ induced by an orbit of $\pi$ is an empty graph.
\end{thm}
\begin{proof}
    First part of Theorem \ref{TW: kiedy TF-izo sa izo za pomoca Im(alpha)} tells us that $\gamma(\pi) = \pi^{-1}$. Lemma \ref{LEMAT: bijekcja Antymorfizmy <-> Swithing elementy} assures us that $x \cdot \pi \in S(\Gamma)$, and therefore $x\cdot \pi$ is a strongly switching element. Then by Lemma \ref{LEM: sigma(h) = h^{-1} daje ze e i h^2 sa izo} and Observation \ref{OBS: every achivable subset of S(H,sigma) is sum of conj classes} we get that $x \cdot \pi^k$ is a strongly switching element for every $k \in \mathbb{Z}$. Now applying  Proposition \ref{PROP: warunek na istnienie grafu TF-izo} for each $\pi^k$ ends the proof.
\end{proof}

One can also prove the above result by using only properties of matrices, as was done in \cite{zbMATH04108921}. Proof presented above is elegant and demonstrates how study of graphs TF-isomorphic to $\Gamma$ can be used in proving properties of $\Gamma$ itself. Two theorems stated above provide rather restrictive conditions for a TF-projection $\pi$ of given graph $\Gamma$, such that there exists a graph with adjacency matrix $A_{\Gamma} \cdot m_{\pi}$. Those conditions can speed up an algorithmic search for graphs TF-isomorphic to a given graph $\Gamma$.

\begin{cor}
    Let $\Gamma$ be a connected, reduced and nonbipartite graph. Then every conjugacy class inside a subset of strongly switching elements of $\Gamma$ is represented by an element of the form $x\cdot \pi_0$, where $\pi_0$ is such that $\text{ord}(\pi_0) = 2^k$ for some positive integer $k$, subgraphs induced by orbits of $\pi_0$ are empty graphs and $\gamma(\pi_0) = {\pi_0}^{-1}$.
\end{cor}
\begin{proof}
    It is enough to combine Theorem \ref{THM: kazdy Main comoponent jest prepezentowany przez element rzedu 2^k} with Theorem \ref{THM: jesli A m_pi jest grafem, to orbity pi to antykliki}.
\end{proof}


We will end this section with construction an infinite family of graphs which have a lot of TF-isomorphic graphs. Those will be Cayley graphs, which have almost the same number of TF-isomorphic graphs, as vertices, and all of those graphs will turn out to be stable. This construction shows, that even if one is only interested in vertex transitive graphs, it is not possible to significantly improve the conjectured upper bound on number of TF-isomorphic graphs (cf. Question \ref{QUE: czy zawsze Aut > TF-iso graphs?}).

\begin{thm}
    If $\mathcal{G}_{\text{TF}}(\text{H},\sigma) = |\text{H}|$, then there exists $k \in \mathbb{Z}_{\geq 0}$ such that $\text{H} \cong \mathbb{Z}_2^{k}$ and $\sigma \equiv Id$.
\end{thm}
\begin{proof}
    At first let us notice, that in general $|S(\text{H},\sigma)| \leq |\text{H}|$, so $S(\text{H},\sigma) = x\cdot \text{H}$, and therefore $\sigma$ is equal to the inverse function. On the other hand, if $\sigma$ is not the identity, then there exists a nontrivial element in $\text{Im}(\alpha_{\sigma})$, a contradiction. Therefore, $\sigma \equiv Id$, for any $h \in \text{H}$ we get $h = h^{-1}$ and equivalently $h^2 = e$. The fact that last equation holds for any element of $\text{H}$ implies that in fact $\text{H} \cong \mathbb{Z}_2^{k}$.
\end{proof}

To construct family of graphs with large number of TF-isomorphic stable graphs we have to introduce two graphs which will become important in our construction.

\begin{defi}
    Let $k\geq 13$ be a positive integer. We define $M(k)$ to be a graph with vertex set $\{1, \ldots, k\}$ and edge set $E_{M(k)}$, where
    $$
    E_1 = \{(1,2),(1,3),(2,4),(3,5),(4,7),(5,6),(6,8),(7,8),(7,9),(8,10),(9,k)\},
    $$
    $$
    E_2 = \{ (i,i+1) \mid 10 \leq i \leq k-1 \}, \text{ and } E_{M(k)} = E_1 \cup E_2.
    $$
    \center
    \includegraphics[scale = 1]{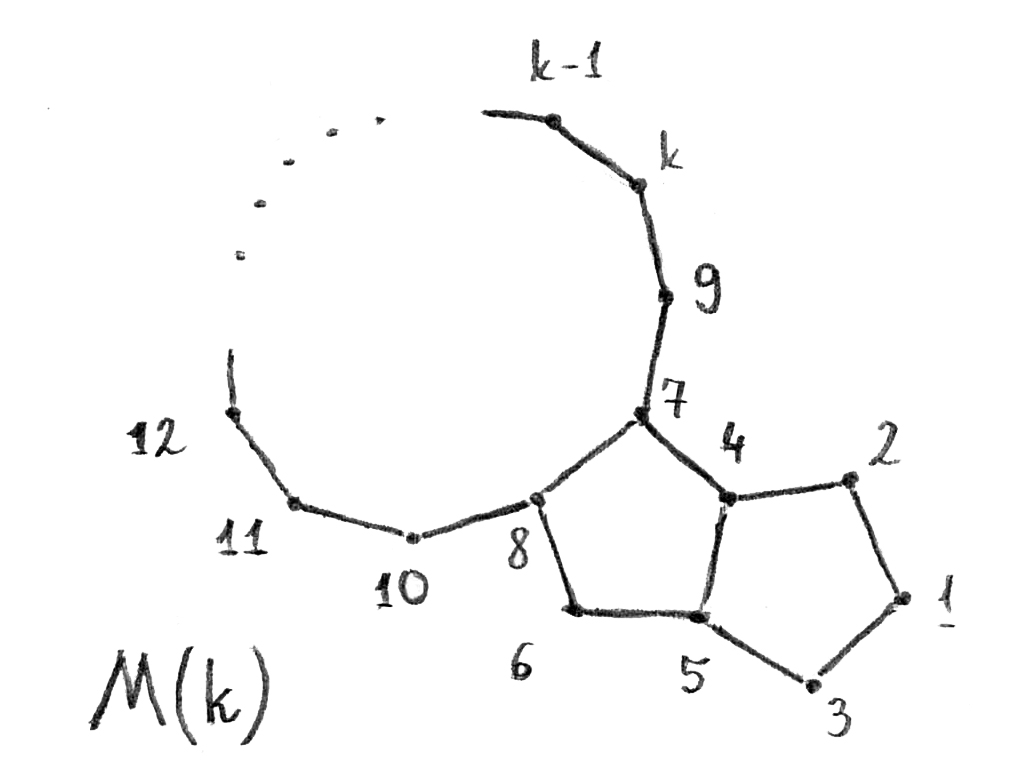}
\end{defi}

\begin{defi}
    For any graph $\Gamma=(V,E)$ a line graph $L(\Gamma)$ of a graph $\Gamma$ is a graph with vertex set $V_{L(\Gamma)} = E$ and edge set
    $$
    E_{L(\Gamma)} = \{((v_1,v_2),(v_3,v_4)) \in E \times E \mid v_1 = v_3, v_1 = v_4, v_2 = v_3 \text{ or } v_2 = v_4 \}.
    $$
\end{defi}

\begin{defi}
    Let $k\geq 13$ be a positive integer. We define $M_0(k)$ to be a graph with
    $$
    \text{vertex set } V_{M_0(k)} = \{(i,i) \mid 1 \leq i \leq k\} \cup E_{M(k)} \text{ and edge set constructed as below}
    $$
    $$
    E_1 = \{((i,i),(j,j)) \in  V_{M_0(k)} \times V_{M_0(k)} \mid (i,j) \in E_{{M(k)}^2}\}, \text{ } E_2 = \{ ((i,i),(i,j)) \in V_{M_0(k)} \times V_{M_0(k)} \},
    $$
    $$
    E_3 = \{ ((i,i),(m,n)) \in V_{M_0(k)} \times V_{M_0(k)} \mid (i,m) \in E_{M(k)} \text{ or } (i,n) \in E_{M(k)} \},
    $$
    $$
    E_4 = \{ ((i,j),(m,n)) \in  V_{M_0(k)} \times V_{M_0(k)} \mid  ((i,j),(m,n)) \in E_{L(M(k))} \}, \text{ and } E_{M_0(k)} = E_1 \cup E_2 \cup E_3 \cup E_4.
    $$
    \center 
    \includegraphics[scale = 0.7]{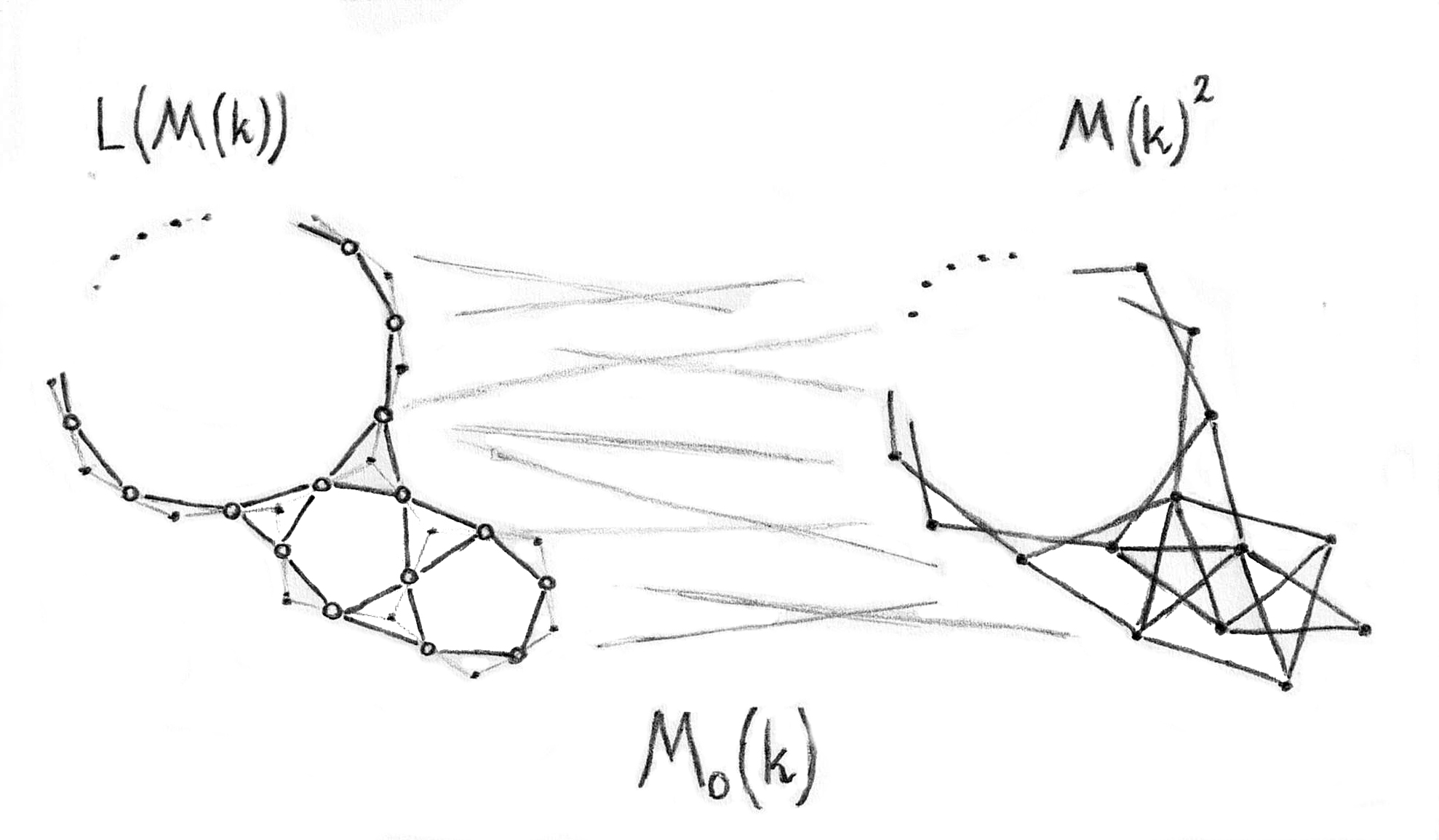}

\end{defi}

\begin{lem} \label{LEM: graph M_0(k) is asymmetric}
    For any positive integer $k \geq 13$, graph $M_0(k)$ is asymmetric. 
\end{lem}
\begin{proof}
Let $\pi_0$ be any automorphism of $M_0(k)$. We will show that $\pi_0$ is the identity. One can check, that the only vertices of degree $11$ in $M_0(k)$ are $(5,5)$ and $(8,8)$ and the only vertices of degree $13$ are $(4,4)$ and $(7,7)$. 
\newline
Let us first show that $(5,5)$ and $(8,8)$ are fixed points. If not, then $\pi_0((5,5)) = (8,8)$ and $\pi_0((8,8)) = (5,5)$. Then $\pi_0((4,4)) = (7,7)$ since an image of $(4,4)$ by $\pi_0$ have to be of degree $13$ and have to be connected to $\pi_0((5,5))$. Similarly $\pi_0((7,7)) = (4,4)$. Now notice, that $(1,1)$ is a vertex of degree $6$ connected to both $(4,4)$ and $(5,5)$. Therefore $\pi_0(1,1)$ is a vertex of degree $6$ connected to both $(7,7)$ and $(8,8)$, however such vertex does not exist, which gives a contradiction. This shows that $(5,5)$ and $(8,8)$ are fixed points, moreover $(4,4)$ and $(7,7)$ also are as both are unique vertices of degree $13$ connected to $(8,8)$ or $(5,5)$ respectively. 
\newline
Now we shall note that the only vertices of degree $10$ are $(6,6)$, $(4,5)$, $(4,7)$ and $(7,8)$. Even more observe that $(6,6)$ is the only vertex of degree $10$ which is not connected to any vertex of degree $11$, and therefore it is a fixed point of $\pi_0$. Now note, that from the set $\{(4,5),(4,7),(7,8)\}$ of other vertices of degree $10$, $(7,8)$ is the only vertex which is not connected to $(5,5)$, and $(4,5)$ is the only vertex which is not connected to $(8,8)$. Therefore $(4,5)$, $(7,8)$, and consequently $(4,7)$ are also fixed points of $\pi_0$.
\newline
By continuing this argument of checking if some vertex is unique with given degree and connections to vertices which are already known to be fixed points one can prove (in order as they are listed) that
$(2,2), (1,1), (3,3), (9,9), (10,10), (1,2), (1,3), (2,4), (3,5), (5,6), (6,8), (7,8),$ 
$(7,9), (8,10)$ and $(10,11)$ are all fixed points of $\pi_0$.
\newline
Now we prove by induction that $(i,i)$ and $(i,i+1)$ for $11\leq i \leq k-1$ are fixed points of $\pi_0$. Assume that this holds for $i_0 - 1$. Then $(i_0,i_0 + 1)$ is the only vertex which is not yet proven to be a fixed point which is connected to both $(i_0-1,i_0-1)$ and $(i_0,i_0)$. This proves that $(i_0,i_0 + 1)$ indeed is a fixed point. Now observe that $(i_0+1,i_0+1)$ is the only vertex which is not yet proven to be a fixed point, which is connected to both $(i_0-1,i_0-1)$ and $(i_0,i_0+1)$. This shows that $(i_0+1,i_0+1)$ also is a fixed point, which finishes the inductive argument.
\newline
The only vertex we had not yet proven to be a fixed point is $(9,k)$, however a permutation $\pi_0$ cannot fix all but one vertex, so indeed $\pi_0$ is an identity, hence the proof is finished.
\end{proof}

Now we are ready for the construction. We should note that this construction is based on one from \cite{zbMATH03291926}.

\begin{defi}
    A Cayley graph of a finite group $\text{H}$ with the set $S \subseteq \text{H} \backslash \{e\}$ such that $S^{-1} = S$ is a graph $\text{Cay}(\text{H},S) = (V,E)$ in which 
    $$
    V = \text{H} \text{ and }
    E = \{(h_1,h_2) \in \text{H} \times \text{H} \mid h_1^{-1} h_2 \in S \}.
    $$
    A Graphical Regular Representation (GRR) of a finite group $\text{H}$ is a Cayley graph $\Gamma = \text{Cay}(\text{H},S)$ such that $\text{Aut}(\Gamma) \cong \text{H}$.
\end{defi}

We say that a graph $\Gamma$ have valency $k_0$ if every vertex of this graph have degree $k_0$.

\begin{thm} \label{THM: konstrukcja grafu z wieloma TF-izomorficznymi - GRR dla Z_2 fo k}
    For every positive integer $k \geq 13$ there exists a connected, reduced, nonbipartite graph $\Gamma$ which is a GRR for a group $\mathbb{Z}_2^{k}$, have valency $2k + 2$ and there exists $2^k - 2k - 2$ pairwise nonisomorphic graphs, which are TF-isomorphic to $\Gamma$. 
\end{thm}
\begin{proof}
    Let $\mathbb{Z}_2^k$ be generated by $x_1, \ldots, x_k$. Let $S = \{x_1,\ldots,x_k\} \cup \{x_i \cdot x_j \mid (i,j) \in E_{M(k)}\}$. We will prove that $\Gamma = \text{Cay}(\mathbb{Z}_2^k,S)$ is a graph satisfying  statement of the theorem. It is easy to see that $\Gamma$ is connected and nonbipartite, even more we have $\text{B}\Gamma) \equiv \text{Cay}(\mathbb{Z}_2^{k+1},  x_{k+1} \cdot S)$ where $x_{k+1}$ is the additional generator of $\mathbb{Z}_2^{k+1}$. Since $\Gamma$ is connected and nonbipartite, $\text{B}\Gamma$ is connected. We will prove that $\mathbb{Z}_2^{k+1}$ is full the automorphisms group for $\text{B}\Gamma$. Let $\pi_0$ be an element of $\text{Aut}(\text{B}\Gamma)$ which maps vertex $e$ to vertex $g$. Then $\pi_1 = (g \cdot )^{-1} \circ \pi_0$ is an automorphism of $\text{B}\Gamma$ which fixes $e$. We will show that if a vertex is a fixed point of an automorphism $\pi_1$, then every vertex in its neighbourhood is a fixed point. This will complete the proof, as $\text{B}\Gamma$ is connected.
    \newline
    Let $g_0$ be a vertex of $\text{B}\Gamma$ which is fixed by $\pi_1$. Then $\pi_1$ gives a permutation on vertices from $N(g_0)$. Let us now consider an a graph $G(g_0)$ with vertex set $N(g_0)$ and edges between vertices $g_1 \neq g_2 \in N(g_0)$ if and only if there exists more than two paths of length two between those vertices in the graph $\text{B}\Gamma$. Since $\pi_1$ is an automorphism, number of paths of length two between vertices $g_1$ and $g_2$ is the same as number of paths of length two between $\pi_1(g_1)$ and $\pi_1(g_2)$. Therefore $\pi_1 {|}_{N(g_0)}$ is an automorphism of $G(g_0)$. To end the proof we will determine set of edges of $G(g_0)$ and then prove that it is asymmetric.
    \newline
    A path of length two is created between vertices $g_1$ and $g_2$ if there exists such $s_1,s_2 \in x_{k+1} S$ such that $s_1 s_2 g_1 = g_2$. Therefore without loss of generality look at $G(e)$, which is isomorphic to every other graph $G(g_0)$. Since every element of $x_{k+1} S$, written down in terms of generators $x_1,\ldots, x_{k+1}$ admits $x_{k+1}$ and ${x_{k+1}}^2 = e$, we can simply ignore it, and label all vertices of $G(e)$ by vertices of $S$. Further we connect vertices labeled by $u$ and $v$ with edge if and only if there exists more than two ordered pairs $(s_1,s_2) \in S \times S$ such that $s_1 s_2 u = v$, which is equivalent to $s_1 s_2 = u v$. We refer to this relabeled graph as $G'$. We call the set $\{x_1,\ldots, x_{k+1}\}$ a standard basis of $\mathbb{Z}_2^{k+1}$.
    \newline
    Let $x_i$ and $x_j$ be two vertices of $G'$. Since both of those have one nonzero power in standard basis, both $s_1$ and $s_2$ consist of the same number of nontrivial powers in standard basis. If $s_1 = x_p$ and $s_2 = x_q$ for $p,q \leq k$, then one of those have to cancel $x_i$ and other have to give a nonzero power in $x_j$, Therefore only possible pairs are $(x_i,x_j)$ and $(x_j,x_i)$. It gives two trivial pairs. If $s_1 = x_p x_q$ and $s_2 = x_r x_s$, then $x_i x_j = x_p x_q x_r x_s$, so one of numbers $p,q,r,s$ have to be equal to $i$, and another one to $j$. If either of $s_1, s_2$ is equal to $x_i x_j$, then other one have to be equal to identity, which is a contradiction. The only other possibility is that up to permutation of $s_1,s_2$, $s_1 = x_i x_p$ and $s_2 = x_j x_p$, which is equivalent to the fact that there is a path of length two in a graph $M(k)$ between vertices $i$ and $j$.
    \newline
    Let now $x_i x_j$ and $x_i x_m$ be two vertices of $G'$. Then we have pairs $(x_i x_j,x_i x_m)$, $(x_i x_m, x_i x_j)$, $(x_j,x_m)$, hence $x_i x_j$ and $x_i x_m$ are connected in $G'$.
    \newline
    Let now $x_i x_j$ and $x_m x_n$ be two vertices of $G'$. We get $s_1 s_2 = x_i x_j x_m x_m$, and since both $s_1$ and $s_2$ consists of at most two elements of the standard basis, equality holds, and up to permutation of $i$ and $j$ and up to permutation of $m$ and $n$ either $s_1 = x_i x_j$ and $s_2 = x_m x_n$ or $s_1 = x_i x_m$ and $s_2 = x_j x_n$. First case produces two pairs $(s_1,s_2)$, and if the second case holds, then $i,j,n,m$ is a four cycle in the graph $M(k)$, however, this graph does not have four cycles, and therefore there are no edges between elements of the form $x_i x_j$ and $x_m x_n$.
    \newline
    Let now $x_i x_j$ and $x_i$ be two vertices of $G'$. Since every vertex in $M(k)$ have degree at least two, there exists $m$ such that both $x_m$ and $x_m x_j$ are elements of $S$. Then, pairs $(s_1,s_2) = (x_i x_j,x_i), (x_i,x_i x_j), (x_m, x_m x_j)$ satisfy the condition $s_1 s_2 = (x_i x_j) x_i = x_j$, and therefore $x_i x_j$ and $x_i$ are connected in $G'$.
    \newline
    Let finally $x_i x_j$ and $x_m$ be two vertices of $G'$. Then $s_1 s_2 = x_i x_j x_m$, and since $s_1$ and $s_2$ both consists of at most two elements from the standard basis, up to transpositions of pairs $s_1,s_2$ and $i,j$ we either have $s_1 = x_i x_j, s_2 = x_m$ or $s_1 = x_i x_m, s_2 = x_j$. The first trivial case always produces two pairs $s_1,s_2$. Therefore there is an edge between $x_i x_j$ and $x_m$ if and only if there exists such $m$ that $s_1 = x_i x_m, s_2 = x_j$ holds, which is equivalent to $x_m$ being connected to either $x_i$ or $x_j$.
    \newline
    The above reasoning ensures, that $G'$ is isomorphic to a graph $M_0(k)$ by map
    $$
    x_i \mapsto (i,i), \quad x_i x_j \mapsto (i,j).
    $$
    By Lemma \ref{LEM: graph M_0(k) is asymmetric} $G'$ is asymmetric, therefore $\pi_1$ fixes every element of the set $N(g_0)$, which finally proves that $\text{B}\Gamma$ is a GRR. Therefore $\Gamma$ also is a GRR and it is stable with $\text{Aut}^{\pi}(\Gamma) \cong \mathbb{Z}_2^k$.
    \newline
Since $\text{Aut}(\text{B}\Gamma) \cong \mathbb{Z}_2^{k+1}$, we have $S(\Gamma) = x_{k+1} \cdot \text{Aut}(\Gamma)$. $\text{Aut}(\text{B}\Gamma)$ is abelian, therefore each conjugacy class contains exactly one element. To compute number of pairwise nonisomorphic graphs TF-isomorphic to $\Gamma$ we have to obtain the number of strongly switching elements. Applying Lemma \ref{LEMAT: bijekcja Antymorfizmy <-> Swithing elementy} and Proposition \ref{PROP: warunek na istnienie grafu TF-izo} ensures us, that number of such elements is equal to the number of elements from $\text{Aut}(\Gamma) \cong \mathbb{Z}_2^k$ which maps each vertex to one outside of its neighbourhood. Since $\Gamma$ is a Cayley graph, it is equivalent to the fact that $e$ is mapped to an element outside of $S$. A graph $M(k)$ have $k$ vertices and $k+2$ edges, so $|S| = 2k+2$, and therefore $|\mathbb{Z}_2^k \text{ }\backslash \text{ } S| = 2^k - 2k - 2$. This shows that there exists exactly $2^k - 2k - 2$ pairwise nonisomorphic graphs, all of which are TF-isomorphic to $\Gamma$. By Proposition \ref{PROP: co sie dzieje z gammą przy przeksztalacaniu grafu na TF-izo} all of those graphs are stable Cayley graphs.
\end{proof}

We finish the paper with some open problems which may stimulate further research. 
We have showed that for every finite group $\text{H}$ and $\sigma \in \text{Aut}(\text{H})$ such that $\sigma^2 \equiv Id$ there is a $(\text{H},\sigma)$-graph. This is not always true for graphs $\Gamma$ such that the group $\text{Aut}^{\pi}(\Gamma)$ acts transitively on vertices. Therefore one can ask the following
\begin{que}
    For which $\text{H}$ and $\sigma \in \text{Aut}(\text{H})$ such that $\sigma^2 \equiv Id$, there exists a $(\text{H},\sigma)$-graph $\Gamma$ such that $\text{Aut}^{\pi}(\Gamma)$ acts transitively on vertices of $\Gamma$?
\end{que}
In Section \ref{SECTION: Further notes on nonisomorphic graphs which are TF-isomorphic to a given graph} we managed to reduce Question \ref{QUE: czy zawsze Aut > TF-iso graphs?} to a group theoretic one

\begin{que}
Let $P_2$ be a finite $2$-group and $\sigma \in \text{Aut}(\text{P}_2)$ be an automorphism such that $\sigma^2 \equiv Id$. Is it true that $|\text{Fix}(\sigma)| \geq \mathcal{G}_{\text{TF}(\text{P}_2,\sigma)}$?
\end{que}
Here $\mathcal{G}_{\text{TF}(P_2,\sigma)}$  denotes the number of conjugacy classes of a group $P_2 \rtimes_{\theta} \mathbb{Z}_2$ which do not belong to $P_2 \times \{e\}$ and are of order two. The semidirect product is defined by $\theta(x) \equiv \sigma$.


\section*{Acknowledgments}

I would like to express my gratitude to Dorde Mitrovic for providing me literature connected to the subject. Prior to that, I based my research only on my results. I am grateful to Tomasz Kowalczyk for asking a question which eventually became Theorem \ref{konstrukcja (H,sigma) grafow}, also  for helping me to write down my results in a more professional manner. I shall also thank Jakub Gaj for translating the paper \cite{zbMATH03291926}. It became crucial to the construction in the Theorem \ref{THM: konstrukcja grafu z wieloma TF-izomorficznymi - GRR dla Z_2 fo k}.


\addtolength{\leftmargin}{0cm}
\setlength{\itemindent}{0cm}
\bibliographystyle{plain}
\bibliography{refs.bib}

\begin{thebibliography}{10}

\bibitem{zbMATH03192673}
P{\'a}l Erd{\H{o}}s and Alfr{\'e}d R{\'e}nyi.
\newblock Asymmetric graphs.
\newblock {\em Acta Math. Acad. Sci. Hung.}, 14:295--315, 1963.

\bibitem{zbMATH06719978}
Richard~H. Hammack and Cristina Mullican.
\newblock Neighborhood reconstruction and cancellation of graphs.
\newblock {\em Electron. J. Comb.}, 24(2):research paper p2.8, 11, 2017.

\bibitem{zbMATH07081656}
Ademir Hujdurovi{\'c}.
\newblock Graphs with {Cayley} canonical double covers.
\newblock {\em Discrete Math.}, 342(9):2542--2548, 2019.

\bibitem{zbMATH03291926}
W.~Imrich.
\newblock Graphen mit transitiver {Automorphismengruppe}.
\newblock {\em Monatsh. Math.}, 73:341--347, 1969.

\bibitem{zbMATH05944472}
J.~Lauri, R.~Mizzi, and R.~Scapellato.
\newblock Two-fold automorphisms of graphs.
\newblock {\em Australas. J. Comb.}, 49:165--176, 2011.

\bibitem{zbMATH07349648}
J.~Lauri, R.~Mizzi, and R.~Scapellato.
\newblock The construction of a smallest unstable asymmetric graph and a family of unstable asymmetric graphs with an arbitrarily high index of instability.
\newblock {\em Discrete Appl. Math.}, 266:85--91, 2019.

\bibitem{zbMATH06471221}
Josef Lauri, Russell Mizzi, and Raffaele Scapellato.
\newblock Unstable graphs: a fresh outlook via {TF}-automorphisms.
\newblock {\em Ars Math. Contemp.}, 8(1):115--131, 2015.

\bibitem{zbMATH04108921}
D.~Maru{\v{s}}i{\v{c}}, R.~Scapellato, and N.~Zagaglia Salvi.
\newblock A characterization of particular symmetric (0,1) matrices.
\newblock {\em Linear Algebra Appl.}, 119:153--162, 1989.

\bibitem{zbMATH00064637}
Dragan Maru{\v{s}}i{\v{c}}, Raffaele Scapellato, and Norma Zagaglia~Salvi.
\newblock Generalized {Cayley} graphs.
\newblock {\em Discrete Math.}, 102(3):279--285, 1992.

\bibitem{pr1997}
Walter Pacco and Raffaele Scapellato.
\newblock Digraphs having the same canonical double covering.
\newblock {\em Discrete Math.}, 173(1-3):291--296, 1997.

\bibitem{zbMATH01675824}
David~B. Surowski.
\newblock Stability of arc-transitive graphs.
\newblock {\em J. Graph Theory}, 38(2):95--110, 2001.

\bibitem{zbMATH05249593}
Steve Wilson.
\newblock Unexpected symmetries in unstable graphs.
\newblock {\em J. Comb. Theory, Ser. B}, 98(2):359--383, 2008.

\bibitem{zbMATH03353315}
Bohdan Zelinka.
\newblock The group of autotopies of a digraph.
\newblock {\em Czech. Math. J.}, 21:619--624, 1971.

\bibitem{zbMATH03378944}
Bohdan Zelinka.
\newblock Isotopy of digraphs.
\newblock {\em Czech. Math. J.}, 22:353--360, 1972.

\bibitem{zbMATH03559606}
Bohdan Zelinka.
\newblock A remark on isotopies of digraphs and permutation matrices.
\newblock {\em {\v{C}}as. P{\v{e}}stov{\'a}n{\'{\i}} Mat.}, 102:230--233, 1977.

\end{thebibliography}




\end{document}